\documentclass[11p]{article}

\usepackage{xspace}
\usepackage{url}

\usepackage{amsmath,amsthm,verbatim,amssymb,amsfonts,amscd, graphicx,color}
\newtheorem{Theorem}{Theorem}

\newtheorem{Lemma}[Theorem]{Lemma}

\newtheorem{Remark}{Remark}
\newtheorem{Definition}{Definition}
\newtheorem{Assumption}{Assumption}

\newcommand{\Uad}{\mathcal{U}}
\newcommand{\Pb}{\mbox{\rm (P)}\xspace}

\newcommand{\dx}{\,\mathrm{d}x}
\newcommand{\dt}{\,\mathrm{d}t}

\newcommand{\llangle}{\big\langle}
\newcommand{\rrangle}{\big\rangle}

\newcommand{\bd}{\begin{displaymath}}
\newcommand{\ed}{\end{displaymath}}
\newcommand{\be}{\begin{equation}}
\newcommand{\ee}{\end{equation}}
\newcommand{\bea}{\begin{eqnarray}}
\newcommand{\eea}{\end{eqnarray}}
\newcommand{\bda}{\begin{eqnarray*}}
\newcommand{\eda}{\end{eqnarray*}}
\newcommand{\ba}{\begin{array}}
\newcommand{\ea}{\end{array}}

\renewenvironment{proof}[1]{%
\par\vspace{1\baselineskip}%
\noindent{\em Proof#1.\ \ }\ignorespaces }{%
\nobreak\hfill\mbox{\ \ $\Box$}%
\par\vspace{1\baselineskip}%
}

\newcommand{\A}{{\mathcal A}}

\newcommand{\sth}{ \, :\;}

\newcommand{\dd}{\mbox{\rm\,d}}

\newcommand{\Z}{{\cal Z}}

\def\R{\mathbb{R}}
\newcommand{\reff}{\eqref}
\newcommand{\bino}{\bigskip\noindent}

\numberwithin{equation}{section}
\title{Solution stability of parabolic optimal control problems with fixed state-distribution of the controls\thanks{
The research presented in this paper is supported by the Austrian Science Foundation (FWF) under grant No I4571.}}

\author{Alberto Dom\'inguez Corella\footnote{Institute of Statistics and Mathematical Methods in Economics, 
Vienna University of Technology, Austria, {\tt alberto.corella@tuwien.ac.at}}
\and Nicolai Jork\footnote{The same affiliation, {\tt nicolai.jork@tuwien.ac.at}}
\and Vladimir M. Veliov\footnote{The same affiliation, {\tt vladimir.veliov@tuwien.ac.at}} }

\date{}

\begin{document}

\maketitle

\begin{abstract}
The paper presents results about strong metric subregularity of the optimality mapping associated with the system
of first-order necessary optimality conditions for a problem of optimal control of a semilinear parabolic equation. 
The control has a predefined spatial distribution and only the magnitude at any time is a subject of choice.
The obtained conditions for subregularity imply, in particular, sufficient optimality conditions that 
extend the known ones.  

The paper is complementary to a companion one by the same authors, in which a distributed control is considered.
\end{abstract}

\section{Introduction} \label{SIntro}

Let $T\in \mathbb R$ and let $\Omega\subset\mathbb R^n$, $1\le n\le 3$, be a bounded domain with Lipschitz boundary 
$\partial \Omega$. Denote by $Q:=\Omega\times (0,T)$ the space-time cylinder and by 
$\Sigma:=\partial \Omega \times (0,T)$ its lateral boundary.
In the present paper, we investigate the following optimal control problem:
	\begin{equation}
 		  \Pb \  \ \min_{u \in \Uad}\bigg\{ J(u) := \int_Q [L_0(x,t,y(x,t))+\langle L_1(x,t,y(x,t)),u(t) \rangle] \dx\dt\bigg\},
	\label{ocp1}
	\end{equation}
subject to
\be
		\left\{ \begin{array}{ll}
		\frac{\partial y}{\partial t}+\mathcal Ay+f(x,t,y) =\langle g(x), u(t)\rangle  &\text{ in }\ Q,\\
		\qquad\; y=0 \text{ on } \Sigma,\quad y(\cdot,0)  =  y_0\ &\text{ on } \Omega.\\
	\end{array} \right.
	\label{see1}
\ee
Here $y: Q \to \R$ is the state, $u : [0,T] \to \R^m$, is the control, $m\in \mathbb{N}$, $\langle \cdot, \cdot \rangle$ is the scalar 
product in $\R^m$, the functions $L_0, L_1, f, g$ are of corresponding dimensions, $\A$ is an elliptic operator.
Moreover, $g:=(g_1,...,g_m)$ with $g_j\in L^\infty(\Omega)$ satisfies
$\textrm{supp}( g_j)\cap \textrm{supp}( g_i)  =\emptyset$ for all $i,j=1,...,m$, $i \not= j$ 
and $\text{meas}(\textrm{supp}( g_i))>0$ for at least one $i$. The set of admissible controls is 
\be \label{const}
		\Uad := \{u \in L^r(0,T)^m \vert \ u_{a,j} \le u_j \le u_{b,j}\ \text{ for a.e. } t \in [0,T],\ 1\leq j\leq m\},
\ee
where $u_a, u_b\in L^\infty(0,T)^m$ and $ u_{a,j}(t) < u_{b,j}(t) $ a.e. in $[0,T]$, $1\leq j\leq m$.

\bino
In the stability analysis and for approximation methods for optimization problems, in general, an important 
role is played by several regularity properties of the system of first order necessary optimality conditions, see e.g. \cite{DR}. 
The {\em Strong Metric subRegularity} (SMsR) property, \cite{DR,CDK},
of the mapping associated with this system, the so-called {\em optimality mapping}, is especially relevant to the
analysis of numerical methods. This property of the optimality mapping associated with problem \reff{ocp1}--\reff{const}
is the subject of investigation of the present paper. 

Sufficient conditions for the SMsR property are usuallly formullated as strong positive definitness (coercivity) 
of the second derivative
of the objective functional with respect to feasible control variations (or on the so-called critical cone) with respect to the 
$L^2$-norm of the controls. Conditions of this type are also sufficient for optimality.
In the paper, we present several sufficient conditions for SMsR of the optimality mapping of problem \reff{ocp1}--\reff{const},
combining in a {\em unified} way strong and weak coercivity requirements relative to the $L^1$-norm. 
Due to the affine structure of the problem with respect to the control, the conditions involve simultaneously 
the first and the second derivative 
of the objective functional. The importance of including the first derivative in the coercivity condition in $L^1$ 
is known from the existing works on ODE affine optimal control problems (see e.g. \cite{OSVE2}): . Moreover, 
the coercivity condition involves not only a quadratic function of the $L^1$-norm of the 
control variation; instead it involves  a more general homogeneous function of second order jointly depending on the control 
and the corresponding state variation, therefore we call it ``unified''.

The sufficient conditions for SMsR are proved in the paper to imply sufficiency of the first order optimality condition
(the Pontryagin principle). Moreover, these conditions are then equivalently reformulated in terms of several 
``critical cones'' that appear in the literature (see e.g. \cite{Casas-Mateos2020}), showing the generality of the former.

In the recently submitted companion paper \cite{CJV2022b}, we consider a similar problem 
where the control, $u(x,t)$, depends on the spacial position $x$ and the time.  In the present paper, following \cite{CM2021},
the control function $u(t)$ depends only on the time, and each control component $u_j(t)$ has a fixed spacial 
distribution given by the function $g_j(x)$, $j = 1, \ldots, m$. For reader's convenience, 
here we repeat several auxiliary results from \cite{CJV2022b} in a slightly modified form. 
The main results---the strong subregularity theorems in Section \ref{SSR}---are also similar to the ones in \cite{CJV2022b}. 
However, there are important differences: (i) the objective functional is more general (in the companion paper, it is essential that 
the function $L_1$ in the objective functional is affine in $y$ or even independent of $y$ in some of the results); 
(ii) the hierarchy of the sufficient conditions for optimality and subregularity introduced in Section \ref{S_SC} 
is similar to that in the elliptic case. However, this hierarhy is not true for parabolic problems  with controls 
depending on space and time;
(iii) in contrast to the present paper, several of the results
about SMsR in \cite{CJV2022b} have the weaker form of H\"older SMsR. 
We refer to the companion paper \cite{CJV2022b} for comprehensive 
discussions about the relationship between conditions for SMsR, second order sufficient optimality conditions, 
and stability analysis of optimal control problems for elliptic and parabolic equations, which we do not repeat here.

The optimal control problem considered in this paper resembles the one in \cite{CMR}\footnote{
We are thankful to Eduardo Casas, who brought to our attention the problem with control depending only on time.}. 
First order Pontryagin type necessary optimality conditions, as well as second order sufficient optimality 
conditions for strong local minimum, are established in this paper. 
In the present paper, we build upon a priori estimates for the linearized 
states established in \cite{CMR} and study metric subregularity of the optimality mapping, hence also 
stability of the solution. 

\bino
The paper is organized as follows.  Section \ref{Sprel} presents notations, assumptions and known facts about 
semilinear parabolic equations. Preliminary results about the optimal control problem \reff{ocp1}--\reff{const} are given 
in Section \ref{SProblem}. The unified conditions for SMsR are introduced in Section \ref{S_SC} and their sufficiency for optimality 
is discussed. Section \ref{S_SMsR} presents the main results -- two theorems claiming that SMsR property of the
optimality mapping holds under several sets of conditions. Some technical auxiliary results and proofs are given in Appendix.  


\section{Notations, assumptions, and known facts} \label{Sprel}

We begin with some notations and definitions.
Given a non-empty, bounded and Lebesgue measurable set $\Omega\subset \mathbb R^n$, 
we denote by $L^p(\Omega)$, $1\leq p\leq\infty$, the Banach spaces of all measurable 
functions $ \Omega \to  \mathbb R$ for which the usual norm $\| \cdot\|_{L^p(\Omega)}$ is finite.
For a bounded Lipschitz domain $\Omega \subset \mathbb R^n$ (that is, a set with Lipschitz boundary), 
the Sobolev space $H^1_0(\Omega)$ consists of all functions $\Omega \to \R$ 
that have weak first order derivatives in $L^2(\Omega)$ and vanish on the boundary of $\Omega$ (in the trace sense).
The space $H_0^1(\Omega)$ is equipped with its usual norm denoted by $\|\cdot\|_{H^1_0(\Omega)}$. 
By $ H^{-1}(\Omega)$ we denote the topological dual of $H_0^1(\Omega)$,
equipped with the standard norm $\|\cdot\|_{H^{-1}(\Omega)}$.
Given a real Banach space $Z$, the space $L^p(0,T\text{; } Z)$ consist of all strongly measurable functions 
$y:[0,T]\to Z$ that satisfy 
\bda
		\|y\|   _{L^p(0,T\text{; }Z)} &:=& \Big(\int_0^T \|   y(t) \|_{Z}^p\dt\Big)^{\frac{1}{p}}<\infty
                      \qquad \mbox{if } \; 1\leq p<\infty, \\
		\|y\|_{L^\infty(0,T\text{; }Z)} &:=& \text{inf}\{M\in\mathbb{R}\ \vert\ \|y(t)\|_{Z}\leq M \;\text{ for a.e. } t\in (0,T)\}<\infty.
\eda
The Hilbert space $W(0,T)$ consists of all functions in $L^2(0,T\text{; }H^1_0(\Omega))$ that have a distributional derivative 
in $L^2(0,T\text{; }H^{-1}(\Omega))$, that is
	\[
		W(0,T):=\Bigg\{ y\in L^2(0,T;H^1_0(\Omega))\Big\vert  \ \frac{\partial y}{\partial t}\in L^2(0,T;H^{-1}(\Omega)) \Bigg\},
	\]
and is endowed with the norm
	\[
		\|   y\|   _{W(0,T)}:=\|    y\|   _{L^2(0,T;H^1_0(\Omega))}+\|    \partial y/\partial t \|   _{L^2(0,T;H^{-1}(\Omega))}.
	\]
The Banach space $C([0,T]\text{; }L^2(\Omega))$ consists of all continuous functions $y:[0,T]\to L^2(\Omega)$ 
and is equipped with the norm $\max_{t\in [0,T]}\|y(t)\| _{L^2(\Omega)}$. It is well known that $W(0,T)$ is continuously 
embedded in $C([0,T]\text{; }L^2(\Omega))$ and compactly embedded in $L^2(Q)$. We use the notation
$\langle \cdot, \cdot \rangle_X$ for the duality pairing between a Banach space $X$ and its dual.

\bino
The following assumptions, close to those in  
\cite{CDJ2022,Casas-Mateos2020,CM2021,CMR,CT16,CT22,CWW,CWW2018}
are standing in all the paper.

\begin{Assumption} \label{exist.1}
The operator $\mathcal A :H^1_0(\Omega)\to H^{-1}(\Omega)$, is given by
		\[
			\mathcal A =-\sum_{i,j=1}^{n}\partial_{x_j}(a_{i,j}(x)\partial_{x_i} ),
		\]
where $a_{i,j}\in L^{\infty}(\Omega)$ satisfy the uniform ellipticity condition
		\[
    \exists \lambda_{\mathcal{A}}>0:\ \lambda_{\mathcal{A}}\vert\xi\vert^2\leq \sum_{i,j=1}^{n} a_{i,j}(x)\xi_{i}\xi_{j} \;\;
               \forall \xi\in \mathbb{R}^n \;\; 	 \textrm{and a.a.}\ x\in \Omega.
		\]
The matrix with components $a_{i,j}$ is denoted by $A$. 
\end{Assumption}     


\begin{Assumption}  \label{exist.2}
For every $y \in \mathbb R$, the functions $f(\cdot,\cdot,y) \in L^{r}(Q)$,  
$L_0(\cdot,\cdot,y) \in L^{1}(Q)$, $L_{1,j}(\cdot,\cdot,y) \in L^{1}(Q)$ $1\leq j\leq m$,  and $y_0 \in L^{\infty}(\Omega)$, where $r$ is a real number satisfying the inequality
\be \label{Ern}
        r> \max \Big\{2,1+\frac{n}{2} \Big\}.      
\ee
For a.e. $(x,t) \in Q$ the first and the second derivatives of $f$,$L_0$ and $L_{1,j}$ with respect to $y$ exist and are locally 
bounded and locally Lipschitz continuous, uniformly with respect to $(x,t) \in Q$. Moreover, 
$ \frac{\partial f}{\partial y}(x,t,y) \geq 0$ for a.e. $(x,t) \in Q$ and for all $y\in \mathbb R$.
\end{Assumption}  



\bino
Next, for the reader's convenience, we remind some facts about linear and semilinear parabolic equations. 

By definition, the function $y$ is a weak solution of  the semilinear parabolic initial-boundary value problem \eqref{see1}
if $y \in W(0,T)$ with $y(\cdot,0)=0$, and 
\begin{align}   \label{lin.1}
		\int_{0}^{T}\llangle \frac{\partial y}{\partial t}+ \mathcal A y, \psi \rrangle_{H^1_0}\,dt
 =-\int_{0}^{T}  \langle f(\cdot,y),\psi\rangle_{L^2(\Omega)}\,\dd t + \int_{0}^{T} \langle h, \psi\rangle_{L^2(\Omega)} \dd t
\end{align}
for all $\psi\in L^2(0,T, H^1_0(\Omega))$, where $h(x,t) :=  \langle g(x), u(t)\rangle$.

A proof of the next theorem can be found in \cite[Theorem 2.1]{CM2021}.
\begin{Theorem} \label{estsemeq}
For any $u\in L^2(0,T)^m$ the initial-boundary value problem \eqref{see1} has a unique weak solution 
$y_u\in W(0,T)$. If $u\in L^r(0,T)^m$ (see \reff{Ern}) then $y_u\in W(0,T)\cap L^\infty(Q)$. 
Moreover, there exists a positive constant $D_{r}$, independent of $u$, $g$, $f$ and $y_0$, such that 
\begin{equation}
	\|y_u \|    _{L^2(0,T;H^1_0(\Omega))}+\|y_u \|   _{L^\infty(Q)} \leq D_{r} \big(\|\langle u, g\rangle \|   _{L^r(Q)}+\|   f(\cdot,\cdot,0) \| _{L^r(Q)}
          +\|    y_0\|   _{L^\infty(\Omega \big)}).	
	\label{semilin}
\end{equation}
Finally, if $u_k\rightharpoonup u$ weakly in $L^{r}(Q)$, then 
\begin{equation}
	\|y_{u_k}-y_u\|_{L^{\infty}(Q)}+\|y_{u_k}-y_u\|_{L^2(0,T;H^1_0(\Omega))}\to 0.
\label{semilinweak}
\end{equation}
\end{Theorem}

Below we remind some results concerning the linearized version of \eqref{see1} and its adjoint equation.

We consider weak solutions (in the same sense as above)
of the following linear parabolic equation:
\be \label{Elin}
		\left\{ \begin{array}{l}
		\frac{\partial y}{\partial t}+\mathcal Ay+\alpha y =h  \quad \text{ in }\ Q,\\
		y=0 \text{ on } \Sigma,\quad y(\cdot,0) =  y_0\quad \text{ on } \Omega.\\
	\end{array} \right.
\ee 

\begin{Lemma} \label{mainex}
Let $0\leq \alpha\in L^\infty(Q)$ be given.
\begin{enumerate}
\item 
For each $h \in L^2(Q)$ equation \eqref{Elin} has a unique weak solution $y_{h}\in W(0,T)$. 
Moreover, there exists a constant $ C_2 >0$  independent of $h$ and $\alpha$ such that
\be \label{wl2}
		\|y_h\| _{L^2(0,T,H^1_0(\Omega))}\leq  C_2  \|h\| _{L^{2}(Q)}.
\ee
\item 
If, additionally, $h \in L^{r}(Q)$ (we remind \reff{Ern}) and $y_0 \in C(\bar Q)$, then the weak solution $y_h$ of \eqref{Elin} 
belongs to $W(0,T)\cap L^\infty(Q)$. Moreover, there exists a constant $C_r>0$ independent of $h$ and  $\alpha$ such that
\be \label{clr} 
		\|y_h\|_{L^2(0,T,H^1_0(\Omega))}+\|y_h\|_{L^\infty(Q)}\leq C_r \| h \| _{L^{r}(Q)}.
\ee
\end{enumerate}
\end{Lemma}

All claims of the lemma are well known,
see \cite[Theorem 3.13, Theorem 5.5]{Troltzsch2010} for the first statements of the two items; for a proof of the independence 
of the constants $ C_2 $ and $C_r$ on $\alpha$ see \cite{CDJ2022} for a linear elliptic PDE of non-monotone type, 
and \cite{CJV2022b} for the parabolic setting.

The differentiability of the control-to-state operator under Assumptions \ref{exist.1} and \ref{exist.2} is well known, 
see \cite[Theorem 2.4]{CMR}.

\begin{Theorem} \label{derivative}
Let $q>1$. The control-to-state operator $\mathcal G: L^q(0,T)^m\to W(0,T)\cap L^{\infty}(Q)$, given by $\mathcal G(u):=y_u$, is of class $C^2$ and for every $u,v,w\in L^q(0,T)$, it holds that 
$z_{u,v}:=\mathcal G'(u)v$ is the solution of 
\begin{align}
	\left\{\begin{array}{l}
		\frac{d z}{dt}+\mathcal Az+f_y(x,t,y_u) z =\langle g,v\rangle \  \text{ in }\ Q,
		\\ z=0 \ \text{ on }\ \Sigma,\ z(\cdot,0)=0\ \text{ on } \Omega,
	\end{array} \right.
	\label{stddt}
\end{align}
and $\omega_{u,(v,w)}:= \mathcal G''(u)(v,w)$ is the solution of
\begin{align}
	\left\{\begin{array}{l}
	\frac{d \omega}{dt}+\mathcal A \omega+f_y(x,t,y_u) \omega = -f_{yy}(x,t,y_u) z_{u,v } z_{u,w} \  \text{ in }\ Q,
	\\ \omega =0 \ \text{ on }\ \Sigma,\ \omega(\cdot,0)=0\ \text{ on } \Omega.
	\end{array} \right.
	\label{stdddt}
\end{align}
\end{Theorem}

\begin{Lemma} (\cite[Lemma 2.5]{CMR})\label{l2l1time}
Let $\alpha_0\in L^\infty(Q)$, $u\in L^1(0,T)^m$, and let $z\in L^\infty(0,T;L^2(\Omega))\cap L^2(0,T;H^1_0(\Omega))$ 
be the solutions of
\begin{equation}
		\left\{ \begin{array}{lll}
		\frac{\partial y}{\partial t}+\mathcal Ay+\alpha_0 y& =\langle g(x),u(t)\rangle  &\text{ in }\ Q,\\
		y=0 \text{ on } \Sigma,\quad y(\cdot,0)& =  y_0\ &\text{ on } \Omega.\\
	\end{array} \right.
\end{equation}
	Then, the following inequality holds:
\begin{equation}\label{l2l1timeest}
		\|z \|_{L^\infty(0,T;L^2(\Omega))}\leq 2\exp (\| \alpha_0\|_{L^\infty(Q)}) 
           \max _{1\leq j \leq m}\|g_j\|_{L^2(\Omega)} \|v\|_{L^1(0,T)^m}.
\end{equation}
\end{Lemma}

\begin{Remark}  \label{rebound}
By the boundedness of $\mathcal U$ in $L^\infty(0,T)^m$ and Theorem \ref{estsemeq}, there exists a constant 
$M_\mathcal U > 0$ such that
\begin{equation}
	\max\{ \| u\|_{L^\infty(0,T)^m},\|y_u\| _{L^\infty(Q))}\}\leq M_{\mathcal U} \quad \forall u\in \mathcal U.
	\label{unibound}
\end{equation}
\end{Remark}

The estimates in the next lemma constitute a key ingredient to derive stability results in the later sections. 
 It extends \cite[Lemma 2.7]{CDJ2022} from elliptic equations to parabolic ones and was proved in \cite{CJV2022b}. 

\begin{Lemma}(\cite[Lemma 5]{CJV2022b})\label{sprep}
The following statements are fulfilled.
\begin{itemize}
\item[(i)] There exists a positive constant $M_2$ such that  for every $ u,\bar u \in \Uad \text{ and }v\in L^r(Q)$
\begin{align}
  \|   z_{u,v} - z_{\bar u,v}\|   _{L^2(Q)}\le M_2\|   y_u - y_{\bar u}\|   _{C(\bar Q)}\|   z_{\bar u,v}\|   _{L^2(Q)}.\label{E33}
\end{align}
\item[(ii)] Let $X=L^\infty(Q)$ or $X=L^2(Q)$. Then there exists $\varepsilon > 0$ such that for every 
$u, \bar u \in \Uad$ with\\ $\|   y_u - y_{\bar u}\|_{L^\infty(Q)} < \varepsilon$ the following inequalities are satisfied
\begin{align}
	&\|   y_u - y_{\bar u}\|_X \le 2\|   z_{\bar u,u - \bar u}\|   _X \le 3\|   y_u - y_{\bar u}\|_X,
	\label{E2.14.2}\\
	&\|   z_{\bar u,v}\|   _X \le 2 \|   z_{u,v}\|   _X \le 3\|   z_{\bar u,v}\|   _X.\label{E2.15.2}
\end{align}
\end{itemize}
\end{Lemma}

\section{The optimal control problem} \label{SProblem}

The optimal control problem \eqref{ocp1}--\eqref{const}  has a global solution due to the linearity
with respect to the control, the compactness and convexity of the admissible control values, and Theorem \ref{estsemeq}
(see e.g. \cite[Theorem 5.7]{Troltzsch2010}).
On the other hand, the semilinear state equation makes the optimal 
control problem possibly nonconvex, therefore it may have local minimizers. 
We recall the following definitions of local optimality. 

\begin{Definition}{\em
The control $\bar u \in \mathcal U$ is called {\em weak local minimizer} of problem  \eqref{ocp1}--\eqref{const}, 
if there exists a number $\varepsilon>0$ such that 
\[
    J( \bar u) \le J(u) \  \ \ \forall u \in \mathcal U \text{ with } \|u-\bar u\|_{L^1(0,T)^m}\le \varepsilon;
\]
 $\bar u \in \mathcal U$ is called {\em strong local minimizer} of \Pb if there exists $\varepsilon >0$ such that
\[
    J( \bar u) \le J(u) \ \ \ \forall u \in \mathcal U \text{ with } \|y_u-y_{\bar u}\|   _{L^{\infty}(Q)}\le \varepsilon.
\]
Moreover, $\bar u \in \mathcal U$ is called {\em  strict (weak or strong) local minimizer} if the above inequalities are
strict for every admissible $u \ne \bar u $. }
\end{Definition}

Due to the compactness of the set of admissible control values, one can equivalently replace the inequality 
$\|u-\bar u\|_{L^1(0,T)^m}\le \varepsilon$ in the definition of weak local optimality with $\|u-\bar u\|_{L^q(0,T)^m}\le \varepsilon$,
where $q$ is any (finite) number $\geq 1$ (see  \cite[Lemma 2.8]{Casas-Mateos2020}).

The analysis below involves first and second order optimality conditions for  problem \eqref{ocp1}--\eqref{const}.
Further, we use the abbreviation
\bd 
		L(x,t,y,u):=L_0(x,t,y)+\langle L_1(x,t,y),u\rangle.
\ed
The next theorem provides a basis for obtaining such conditions.
It is a consequence of Theorem \ref{derivative} and the chain rule, and adapts \cite[Theorem 2.7]{CMR} to 
the more general objective functional considered in the present paper.

\begin{Theorem} \label{T3.1}
Given, $q>1$, the functional $J:L^q(0,T)^m \longrightarrow \mathbb{R}$ is of class $C^2$. Moreover, 
given $u, v, v_1, v_2 \in L^q(0,T)^m$ we have
\begin{align}
     J'(u)v& =\int_Q\Big(\frac{\partial L_0}{dy}(x,t,y_u)+\Big\langle \frac{\partial L_1}{dy}(x,t,y_u),u \Big \rangle \Big)z_{u,v}+
       \langle L_1(x,t,y_u),v\rangle\dx\dt\\
      &= \int_Q\langle p_ug+ L_1(x,t,y_u),v\rangle\dx\dt,\label{E3.4}\\
   J''(u)(v_1,v_2)& = \int_Q\Big[\frac{\partial^2L}{\partial y^2}(x,t,y_u,u) - p_u\frac{\partial^2f}{\partial y^2}(x,t,y_u)\Big]
       z_{u,v_1}z_{u,v_2}\dx\dt\label{E3.5.a}\\
        &+ \int_Q\Big\langle \frac{\partial L_1}{dy}(x,t,y_u), v_2z_{u,v_1}+ v_1z_{u,v_2}\Big\rangle\dx\dt,\label{E3.5}
\end{align}
Here, $p_u \in W(0,T) \cap C(\bar Q)$ is the unique solution of the adjoint equation
\begin{equation}
      \left\{\begin{array}{l}\displaystyle -\frac{d p}{dt}+\mathcal{A}^*p + \frac{\partial f}{\partial y}(x,t,y_u)p=  
    \frac{\partial L_0}{dy}(x,t,y_u)+\Big\langle \frac{\partial L_1}{dy}(x,t,y_u),u\Big\rangle  \;\; {\rm in } \;\; Q,\\ p = 
         0\text{ on } \Sigma, \ p(\cdot,T)=0 \;\; {\rm on } \; \;\Omega,\end{array}\right.
\label{E3.6}
\end{equation}
where $\A^*$ is the adjoint operator to $\A$.
\end{Theorem}

We introduce the Hamiltonian 
$Q\times\mathbb R\times\mathbb R\times \mathbb R^m \ni (x,t,y,p,u) \mapsto H(x,t,y,p,u) \in \mathbb R$
in the usual way:
\begin{align*}
       H(x,t,y,p,u):=L(x,t,y,u)+p(\langle u,g\rangle-f(x,t,y)).
\end{align*}
We denote the derivative at $\bar u$ in direction $v\in L^r(0,T)^m$ of $H$ by 
$\frac {\partial H}{\partial u}(x,t,\bar y,\bar p,\bar u)(v):=
\langle  L_1(x,t,\bar y)+\bar p(x,t) g(x),v(t)\rangle$ and further 
abbreviate $\frac {\partial \bar H}{\partial u}(x,t):=
\frac {\partial H}{\partial u}(x,t,\bar y,\bar p,\bar u)$. Notice that $\frac {\partial H}{\partial u}(x,t,\bar y,\bar p,\bar u)$ is actually
independent of the last argument.
The Pontryagin type necessary optimality conditions for problem \eqref{ocp1}-\eqref{const} 
appearing in the next theorem
are well known (see e.g. \cite{Casas-Mateos2020,CMR,Troltzsch2010}). For a problem with controls depending only on time, 
we refer to \cite[Theorem 3.3]{CMR}.
 
\begin{Theorem}
If $\bar u$ is a weak local minimizer for problem  \eqref{ocp1}-\eqref{const}, 
then there exist unique elements $\bar y, \bar p\in W(0,T)\cap L^\infty(Q)$ such that 
\begin{align}
& \left\{\begin{array}{l} \frac{d \bar y}{dt}+\mathcal{A}\bar y + f(x,t,\bar y) = 
\langle \bar u,g\rangle \text{ in } Q,\\ \bar y = 
0\text{ on } \Sigma, \ \bar y(\cdot,0)=y_0\text{ on } \Omega.\end{array}\right.
\label{E3.7}\\
&\left\{\begin{array}{l}\displaystyle -  \frac{d \bar p}{dt}+\mathcal{A}^*\bar p =
 \frac {\partial H}{\partial y}(x,t,\bar y,\bar p,\bar u) \text{ in } Q,\\ \bar p = 0\text{ on } \Sigma, \ \bar p(\cdot,T)=
0\text{ on } \Omega.\end{array}\right.
\label{E3.8}
\\
&\int_\Omega \frac {\partial H}{\partial u_j}(x,t,\bar y,\bar p,\bar u) \dd x \,(u_j - \bar u_j(t))  \ge 0 \quad 
 \forall j \in \{1, \ldots, m\}, \forall u_j \in [u_{a,j}, u_{b,j}], {\rm \ and \ for \ a.e.\ } t \in [0,T].
\label{E3.9}
\end{align}
\label{pontryagin}
\end{Theorem}

As a consequence of \reff{E3.8}, for any triplet $(\bar y, \bar p, \bar u)$, $j \in \{1, \ldots, m\}$ and for a.e. $t \in [0,T]$ it holds that
\begin{equation*}
     \bar u_j (t) = \left\{ \begin{array}{ll}
                  u_{a,j}(t) &  \mbox{ if } \;  \int_\Omega \frac{\partial \bar H}{\partial u_j}(x,t) \dd x > 0, \\
                  u_{b,j}(t) &  \mbox{ if } \;  \int_\Omega \frac{\partial \bar H}{\partial u_j}(x,t) \dd x < 0.
                         \end{array} \right.
\end{equation*}

\section{Sufficient optimality conditions} \label{S_SC}

In this section, we present a second order sufficient optimality condition, which is a version of 
\cite[Assumption 3]{CJV2022b} adapted to the case of controls depending only on time.
Below, $\bar u$ is an admissible reference control and $\bar y$ is an element of $W(0,T)\cap L^\infty(Q)$
(presumably the solution of \reff{see1}).

\begin{Assumption} \label{A4}
For a  number $k\in \{0,1,2\}$, at least one of the following conditions is fulfilled: 

\smallskip\noindent
($A_k$): There exist constants $\alpha_k,\gamma_k>0$ such that
\begin{equation}  \label{E3.13.1} 
   J'(\bar u)(u - \bar u) + J''(\bar u)(u - \bar u)^2 \ge \gamma_k\|z_{\bar u,u - \bar u}\|^{k}_{L^2(Q)}\|u -\bar u\|^{2-k}_{L^1(0,T)^m}
\end{equation}
for all $u \in \Uad \text{ with } \|   y_u - \bar y\|_{L^\infty(Q)} < \alpha_k$.

\smallskip\noindent
($B_k$): There exist constants $\tilde \alpha_k, \tilde \gamma_k>0$ such that \reff{E3.13.1} holds for all $u \in \Uad$ such 
that $\|u - \bar u\|_{L^1(0,T)^m} < \tilde\alpha_k$.
\end{Assumption}

Assumption \ref{A4}($B_0$)  was first introduced in \cite{OSVE} in the ODE optimal control context, and was extended
to parabolic PDEs in \cite{CDJ2022}, where also ($A_0$) was introduced.



As it is proved in \cite[Proposition 8]{CJV2022b}, for any $k\in\{0,1,2\}$, Assumption ($A_k$) implies ($B_k$);
if $\bar u $ is bang-bang (that is, $\bar u(t) \in \{ u_a(t), u_b(t) \}$ for a.e. $t \in [0,T]$) 
 then assumptions ($A_k$) and ($B_k$) are equivalent.

%


\bino
Next, we obtain growth estimations for the objective functional, which show, in particular, that  
assumptions \ref{A4}($A_k$) and ($B_k$) are sufficient either for strict weak 
or strict strong local optimality, correspondingly.

\begin{Theorem} \label{T3.3}
The following statements hold.
\begin{enumerate}
\item 
Let the function $L_1$ in the objective functional be independent of $y$.
Let $\bar u \in \Uad$ satisfy the optimality conditions \eqref{E3.7}--\eqref{E3.9} 
and Assumption \ref{A4}($A_k$) with some $k\in \{0,1,2\}$. Then, there exist $\varepsilon_k,\kappa > 0$ such that:
\begin{equation} \label{E3.12}
J(\bar u) + \frac{\kappa}{2}\|y_u - \bar y\|^{k}_{L^2(Q)}\|u - \bar u\|   _{L^1(0,T)^m}^{2-k} \le J(u)
\end{equation}
for all $u \in \Uad \text{ such that } \| y_u - \bar y\|_{L^\infty(Q)} < \varepsilon_k$.
\item 
Let the function $L_1$ in the objective functional be affine with respect to $y$.
Let $\bar u \in \Uad$ satisfy the optimality conditions \eqref{E3.7}--\eqref{E3.9} and Assumption \ref{A4}($B_k$) 
with some $k\in \{1,2\}$. Then, there exist $\varepsilon_k,\kappa_k > 0$ such that \eqref{E3.12} holds for all 
$u \in \Uad \text{ such that } \|   u - \bar u\|   _{L^1(0,T)^m} < \varepsilon_k$.
\item 
Let $\bar u \in \Uad$ satisfy the optimality conditions \eqref{E3.7}--\eqref{E3.9} and Assumption \ref{A4}($B_0$). 
Then, there exist $\varepsilon_0,\kappa_0 > 0$ such that \eqref{E3.12} holds for all 
$u \in \Uad \text{ such that } \|   u - \bar u\|   _{L^1(0,T)^m} < \varepsilon_0$.
\end{enumerate}
\end{Theorem}

A proof of Theorem \ref{T3.3} in case of  a less general objective functional can be found in \cite{CJV2022b}. 
It is a consequence of the next two lemmas, which will be used also in Section \ref{SSR}.
The first of them has been proved for various types of objective functionals, see e.g.
\cite[Lemma 6]{CT16},\cite[Lemma 3.11]{CRT15} or \cite[Lemma 10]{CJV2022b}. 
Due to the more general objective functional in the present paper and for readers' convenience 
we present a proof in the Appendix.

\begin{Lemma} Let $\bar u \in \Uad$. The following holds.
\begin{enumerate}
\item For every $\rho > 0$ there exists $\varepsilon > 0$ such that
\begin{align}
 \vert  [J''(\bar u + \theta(u - \bar u)) - J''(\bar u)](u-\bar u)^2\vert \leq \rho \| u - \bar u\|_{L^1( 0,T)^m}^2
\end{align}
 holds for all $u \in \Uad$ with $\| u - \bar u\|_{L^1( 0,T)^m} < \varepsilon$ and every $\theta \in [0,1]$.\label{drei}
\item Let the function $L_1$ in the objective functional be affine with respect to $y$.
For every $\rho > 0$ there exists $\varepsilon > 0$ such that
\begin{align}
 \vert  [J''(\bar u + \theta(u - \bar u)) - J''(\bar u)](u-\bar u)^2\vert \leq \rho\|z_{\bar u,u-\bar u}\|^2_{L^2(Q)}.
\label{zweis}
\end{align}
holds for all $u \in \Uad$ with $\| u - \bar u\|_{L^1( 0,T)^m} < \varepsilon$ and $\theta \in [0,1]$. 
\label{zwei}
\item Let the function $L_1$ in the objective functional be independent of $y$.
For every $\rho > 0$ there exists $\varepsilon > 0$ such that \eqref{zweis} holds
for all $u \in \Uad$ with $\|   y_u - \bar y\|_{L^\infty(Q)} < \varepsilon$ and $\theta \in [0,1]$. \label{eins}
\end{enumerate}
\label{biglemmat}
\end{Lemma}

The next lemma shows that Assumption \ref{A4} implies a growth similar to \reff{E3.12} of the first derivative 
of the objective functional in a neighborhood of $\bar u$. 

\begin{Lemma} \label{good}
The following claims are fulfilled.
\begin{enumerate}
	\item Let the function $L_1$ in the objective functional be independent of $y$.	
Let $\bar u$ satisfy assumption $(A_k)$, for some $k\in \{0,1,2\}$.
	Then, there exist $\bar \alpha_k,\bar\gamma_k > 0$ such that
	\begin{equation}
		J'(u)(u - \bar u) \ge \bar\gamma_k\|   z_{\bar u,u - \bar u}\|   ^{k}_{L^2(Q)}\|u - \bar u\|^{2-k}_{L^1(0,T)^m}
		\label{E4.6}
	\end{equation}
	 for every $u \in \Uad \text{ with } \|   y_u - \bar y\|_{L^\infty(Q)} < \bar \alpha_k$.
	\item Let the function $L_1$ in the objective functional be affine with respect to $y$.  
Let $\bar u$ satisfy assumption $(B_k)$ for some $k\in \{1,2\}$.
	Then, there exist $\bar \alpha_k ,\bar\gamma_k > 0$ such that \reff{E4.6} holds
	for every $u \in \Uad \text{ with }\| u -\bar u\|_{L^1(0,T)^m} < \bar \alpha_k$.
	\item Let $\bar u$ satisfy assumption $(B_0)$.
	Then, there exist $\bar \alpha_0 ,\bar\gamma_0 > 0$ such that \reff{E4.6} holds
	for every $u \in \Uad \text{ with }\| u -\bar u\|_{L^1(0,T)^m} < \bar \alpha_0$.
	\end{enumerate}
\end{Lemma}



\bino
{\large\bf Reformulations of Assumption 3 using cones.}\\
We recall that some of the items in Assumption \ref{A4} can be formulated equivalently by restricting the admissible control
variations $v = u - \bar u$ to appropriate cones.
This applies to ($B_k$) or to ($A_k$) depending on whether the objective functional explicitly depends on the control or not.

Obviously any admissible control variation $v = u - \bar u$, $u \in \mathcal U$, satisfies the conditions
\begin{equation}\label{sign}
v \in L^2(0,T)^m, \quad  v_j(t) \geq 0 \;\textrm{ whenever }\; \bar u_j(t) =u_{a,j}(t) \; \textrm{ and } \;
    v_j(t) \leq 0 \; \textrm{ whenever } \; \bar u_j(t) =u_{b,j}(t).  
\end{equation}   
Then, for $\tau>0$ define 
\begin{align}
D^\tau_{\bar u}& := 
\Big\{v\in L^2(0,T)^m\Big \vert   v\text{ satisfies }\eqref{sign}\text{ and } v_j(x,t)=
0\text{ if }\Big\vert  \frac{\partial \bar H}{\partial u_j}(x,t)\Big \vert   >\tau, \ 1\leq j \leq m\Big\},\\
G^{\tau}_{\bar u}&:=
\Big\{v\in L^2(0,T)^m\Big \vert v\text{ satisfies }\eqref{sign}\text{ and } J'(\bar u)(v)\leq \tau \|z_{\bar u,v} \|_{L^1(Q)}\Big\},\\
E^{\tau}_{\bar u}&:=
\Big\{v\in L^2(0,T)^m\Big \vert   v\text{ satisfies }\eqref{sign}\text{ and } J'(\bar u)(v)\leq \tau \|z_{\bar u,v} \|_{L^2(Q)}\Big\},\\
C^\tau_{\bar u}&:=D^{\tau}_{\bar u}\cap G^{\tau}_{\bar u}  \label{ccone}.
\end{align}
The cones $D^\tau_{\bar u}$, $E^{\tau}_{\bar u} \text{ and } G^{\tau}_{\bar u}$ were introduced in \cite{Casas12,CT16}
as extensions of the usual critical cone. Most recently, the cone $C^\tau_{\bar u}$ was defined in \cite{Casas-Mateos2020} 
and also used in \cite{CM2021}.
In the ODE control literature, a cone similar to $D^{\tau}_{\bar u}$ has been in use for a long time, see \cite{Osmo75}.

\begin{Theorem} \label{equi1}
\begin{enumerate}
\item For $k\in \{0,2\}$,  Assumption \ref{A4}$(B_k)$ is equivalent to the following condition ($\bar B_{k}$):
there exist constants $\alpha_k,\gamma_k,\tau>0$ such that 
\begin{equation} \label{EbBk}
        J'(\bar u)(u - \bar u) + J''(\bar u)(u - \bar u)^2 \ge 
            \gamma_k\|   z_{\bar u,u - \bar u}\|^{k}_{L^2(Q)}\| u - \bar u\|^{2-k}_{L^1(0,T)^m},
\end{equation}
for all $u \in \Uad$ for which $(u-\bar u )\in  D_{\bar u}^{\tau}\text{ and } \|u-\bar u \| _{L^{1}(0,T)^m} < \alpha_k$.
\item 
Let the function $L_1$ in the objective functional be independent of $y$, then Assumption \ref{A4}$(A_2)$ is equivalent 
to the following condition ($\bar A_{2}$):
there exist constants $\alpha_2,\gamma_2,\tau>0$ such that
\begin{equation} \label{EbA2}
    J'(\bar u)(u - \bar u) + J''(\bar u)(u - \bar u)^2 \ge \gamma_2\|   z_{\bar u,u - \bar u}\|^{2}_{L^2(Q)}
\end{equation}
for all $u \in \Uad$ for which $(u-\bar u )\in C^{\tau}_{\bar u}$ and $\| y_u-\bar y\|_{L^\infty(Q)}<\alpha_2$.
\end{enumerate}
\end{Theorem}

The proof goes along 
the lines of  \cite[Corollary 14,15]{CJV2022b}.\\
By Theorem \ref{T3.3}, the conditions \eqref{EbBk} and \eqref{EbA2} constitute sufficient conditions for strict weak 
or strong local optimality. 

Sufficient second order conditions for (local) optimality based on \eqref{sign}-\eqref{ccone} are given 
in \cite{CMR,Casas-Mateos2020, CT16}.
For instance, it was proved in \cite{Casas12,CRT15,CT16} that the condition:
\begin{equation}
    \exists \delta>0, \tau>0 \ \ \mbox{ such that } \ \ J''(\bar u)v^2\ge \delta \|z_{\bar u,v} \|_{L^2(Q)}^2 \ \  \forall v \in G
\label{soclin}
\end{equation}
is sufficient for weak (in the case $G=D^{\tau}_{\bar u}$) or strong (in the case 
$G=E^{\tau}_{\bar u}$) local optimality in the elliptic and parabolic setting.
It was proven in \cite{Casas-Mateos2020}, that \eqref{soclin} with $G=C^\tau_{\bar u}$ is sufficient for 
strong local optimality.
 To obtain and improve stability results, an additional assumption is usually imposed, called the structural assumption. 
Adapted to the problem considered in this paper, it reads
\begin{equation}
        \exists \kappa > 0 \text{ such that }  \ \ \mbox{meas } \Big\{ t \in [0,T] \sth
        \Big\vert \int_\Omega  \frac {\partial \bar H}{\partial u_j}(x,t)\dx \Big\vert   
        \le \varepsilon  \Big\}  \le \kappa\varepsilon\quad \forall \varepsilon > 0, \ j=1,...,m.
\label{structb}
\end{equation} 
It is known that the assumption \eqref{structb} implies that $\bar u$ is of bang-bang type and the existence of a constant 
$\tilde \kappa>0$ such that the following growth property holds:
\begin{equation}
        J'(\bar u)(u-\bar  u)\geq \tilde \kappa \|u-\bar u\|_{L^1(0,T)^m}^2 \; \;  \forall u\in \mathcal U.
\label{growthfirst}
\end{equation}
For a proof see \cite{ASS16}, \cite{OSVE2} or \cite{S2015}. For stability results under these and additional conditions, 
see \cite{CRT15,CT16, CT22, Qui-Wachsmuth2018, CWW, CWW2018, CJV2022b}. %
\begin{Remark}
We compare the items in Assumption \ref{A4} to the ones using \eqref{soclin} and \eqref{structb}.
\begin{enumerate}
\item Assumption \ref{A4}($A_0$) is implied by the structural assumption \eqref{structb} and possible negative curvature 
as in \cite{CWW,CWW2018}. For details see \cite[Theorem~6.3]{DJV2022}.
\item Assumption \ref{A4}($A_1$)  is implied by the structural assumption \eqref{structb} together with
\[
J''(\bar u)(u-\bar u)^2\ge -\tilde \delta \|z_{\bar u,u-\bar u} \|_{L^2(Q)}\|u-\bar u\|_{L^1(0,T)^m}
\] 
for all $u\in \mathcal U$ and any $\tilde \delta> 0$ sufficiently small.
This is clear by Lemma \ref{l2l1time} and \eqref{growthfirst}.
\item By item two in Theorem \ref{equi1}, Assumption \ref{A4}($A_2$) is implied by \eqref{soclin}. 
\end{enumerate}
\end{Remark}

\section{Strong metric subregularity and auxiliary results} \label{SSR}

In this section we study the strong metric subregularity property (SMSr) of the optimality mapping (see, \cite[Section 3I]{DR} or 
\cite[Section 4]{CDK}), beginning with a precise definition of the latter.

\subsection{The optimality mapping}

We begin by defining some mappings used to represent the optimality map in a convenient way. 
This is done by a sight modification of 
\cite[Section 2.1]{DJV2022} and \cite[Section 4.1]{CJV2022b}.
Given the initial data $y_0$ in \eqref{see1}, we define the set
\be \label{EDL}
D(\mathcal L):=\Big\{ y\in W(0,T)\cap L^{\infty}(Q)\Big \vert \ \Big(\frac{d}{dt}+\mathcal A\Big)y \in L^{r}(Q), y(\cdot,0)=y_0 \Big\}.
\ee
To shorten notation, we define $\mathcal L:D(\mathcal L)\to L^{r}(Q)$ by $\mathcal L:=\frac{d}{dt}+\mathcal A$.
Additionally, we define the mapping
$\mathcal L^*:D(\mathcal L^*)\to L^{r}(Q)$ by $\mathcal L^*:=(-\frac{d}{dt}+\mathcal A^*)$, where
\[
D(\mathcal L^*):=\Big\{ p\in W(0,T)\cap L^{\infty}(Q)\Big \vert  \Big(-\frac{d}{dt}+\mathcal A^*\Big)p \in L^{r}(Q), p(\cdot,T)=0 \Big\}.
\]

With the mappings $\mathcal L$ and $\mathcal L^*$, we recast the semilinear state equation \reff{see1}
and the linear adjoint equation \reff{E3.8} in a short way: 
\[
\mathcal L y=\langle u,g\rangle-f(\cdot,\cdot,y),
\]
\[
\mathcal L^* p=\frac {\partial L}{\partial y}(\cdot,\cdot,y_u,u)-p\frac {\partial f}{\partial y}(\cdot,\cdot,y_u)=
\frac {\partial H}{\partial y}(\cdot,\cdot,y_u,p,u).
\]
The normal cone to the set $\mathcal{U}$ at $u \in L^1(0,T)^m$ is defined in the usual way:
\begin{equation*} 
      N_{\mathcal{U}}(u):= \left\{ \begin{array}{cl}
    \big\{\nu \in L^{\infty}(0,T)^m\big\vert  \ \int_0^T \nu (v-u) \dt \le 0 \ \ \forall v \in \mathcal U \big\} & \mbox{ if } u \in \mathcal{U}, \\
    \emptyset  & \mbox{ if } u \not\in \mathcal{U}.
   \end{array}  \right. 
\label{normal}
\end{equation*}
The first order necessary optimality condition for problem \eqref{ocp1}-\eqref{const} in Theorem \ref{pontryagin}
can be recast as 
\begin{align}\label{s6}
\left\{ \begin{array}{cll}
0&=&\mathcal L y+f(\cdot,\cdot,y)-\langle u,g\rangle\\
0&=&\mathcal L^*p- \frac {\partial H}{\partial y}(\cdot,\cdot,y,p,u),\\   
0&\in& 
 \int_\Omega \frac {\partial H}{\partial u}(x, \cdot,y,p,u) \dx + N_{\mathcal U}(u).
\end{array} \right.
\end{align}
For (\ref{s6}) to make sense, a solution $(y,p,u)$ must satisfy $y\in D(\mathcal L)$, $p\in D(\mathcal L^*)$ and 
$u\in\mathcal U$. 
For a local solution $\bar u\in\mathcal U$ of problem (\ref{ocp1})-(\ref{const}), by Theorem \ref{pontryagin}, 
the triple $(y_{\bar u},p_{\bar u},\bar u)$ is a solution of (\ref{s6}).
We define the sets
\begin{align} \label{EYZ}
	\mathcal Y:=D(\mathcal L)\times D(\mathcal L^*)\times\mathcal U\quad\text{and}\quad \mathcal Z:=
L^2(Q)\times L^2(Q)\times L^\infty(0,T)^m,
\end{align}
and consider the set-valued mapping $\Phi:\mathcal Y\twoheadrightarrow\mathcal Z$ given by
\begin{align}\label{optmapping}
\Phi\left( \begin{array}{c}
y \\
p \\
u
\end{array} \right) :=\left( \begin{array}{c}
\mathcal Ly + f(\cdot,\cdot,y)-\langle u,g\rangle \\
\mathcal L^*p- \frac {\partial H}{\partial y}(\cdot,\cdot,y,p,u) \\
 \int_\Omega \frac{\partial H}{\partial u}(x,\cdot,y,p,u) \dx + N_{\mathcal U}(u)
\end{array} \right).
\end{align}
With the abbreviation $\psi := (y,p,u)$, the system (\ref{s6}) can be rewritten as the inclusion $0\in\Phi(\psi)$.
Therefore, the mapping $\Phi:\mathcal Y\twoheadrightarrow\mathcal Z$ is called the {\em optimality mapping} of the optimal 
control problem (\ref{ocp1})-(\ref{const}).
Our goal is to study the stability of system (\ref{s6}), or equivalently, the stability of the solutions of the inclusion 
$0\in\Phi(\psi)$ under perturbations. 
For elements $\xi,\eta\in L^r(Q)$ and $\rho\in L^\infty(0,T)^m$ we consider the perturbed system
\begin{align}\label{s1per}
\left\{ \begin{array}{cll}
\xi&=&-\mathcal Ly+f(\cdot,\cdot,y) - \langle g,u\rangle,\\
\eta&=&-\mathcal L^*p+\frac {\partial H}{\partial y}(\cdot,\cdot,y,p,u),\\
\rho&\in& \int_\Omega \frac{\partial H}{\partial u}(x,\cdot,y,p) \dx+ N_{\mathcal U}(u),
\end{array} \right.
\end{align}
or equivalently, the inclusion $\zeta \in \Phi(\psi)$, where $\zeta := (\xi,\eta,\rho) \in \Z$.


The next theorem is a consequence of the fact that \reff{s1per} represents the Pontryagin maximum principle for an
appropriately perturbed version of problem (\ref{ocp1})-(\ref{const}), for which  a solution exists by the same argument 
as in the beginning of Section \ref{SProblem}.

\begin{Theorem}
For any perturbation $\zeta:=(\xi,\eta,\rho)\in L^r(Q)\times L^r(Q)\times L^\infty(0,T)^m$ there exists 
a triple $\psi := (y,p,u)\in \mathcal Y$ such that $\zeta\in \Phi(\psi)$.
\end{Theorem}

Given a metric space $(\mathcal X,d_{\mathcal X})$, we denote by $B_{\mathcal X}(c,\alpha)$ the closed ball of 
 radius $\alpha>0$ centered at $c\in\mathcal X$.
The spaces $\mathcal Y$ and $\mathcal Z$, introduced in \reff{EYZ}, are endowed with the metrics 
\begin{align} \label{Enzeta}
	d_{\mathcal Y}(\psi_1,\psi_2)&:=\|   y_1-y_2\|   _{L^2(Q)}+\|   p_1-p_2\|   _{L^2(Q)}+\|   u_1-u_2\|   _{L^1(0,T)^m},\\
	\nonumber d_{\mathcal Z}(\zeta_1,\zeta_2)&:=\|\xi_1-\xi_2\| _{L^{2}(Q)}+\|\eta_1-\eta_2\|_{L^{2}(Q)}+
      \|\rho_1-\rho_2\|_{L^\infty(0,T)^m},
\end{align}
where $\psi_i=(y_i,p_i,u_i)$ and $\zeta_i=(\xi_i,\eta_i,\rho_i)$, $i\in\{1,2\}$.
Further on, we denote $\bar\psi:=(y_{\bar u},p_{\bar u},\bar u)$. \\

The following extension of the previous theorem can be proved along the lines of \cite[Theorem 4.12]{DJV2022}.

\begin{Theorem}
Let Assumption \ref{A4}$(A_0)$ hold. For each $\varepsilon > 0$ there exists $\delta > 0$ 
such that for every $\zeta \in B_{\mathcal Z} (0;\delta)$
there exists $\psi \in B_{\mathcal Y}(\bar \psi; \varepsilon)$ satisfying the inclusion $\zeta \in \Phi(\psi)$.
\end{Theorem}


\subsection{Strong metric subregularity: main result} \label{S_SMsR}

This subsection contains one of the main results in this paper: estimates of the difference between the solutions 
of the perturbed 
system \reff{s1per} and a reference solution of the unperturbed one, \reff{s6}, by the size of the perturbations. 
This will be done using the notion of {\em strong metric subregularity} recalled in the next paragraphs.

\begin{Definition}\label{Dsmsr}
Let $\bar \psi$ satisfy $0 \in \Phi(\bar \psi)$.
We say that the optimality mapping $\Phi:\mathcal Y\twoheadrightarrow \mathcal Z$ is 
{\em strongly metrically subregularity} (SMsR) at $(\bar\psi,0)$ if there exist positive numbers $\alpha_1,\alpha_2$ 
and $\kappa$ such that 
	\begin{align*}
	d_{\mathcal Y}(\psi,\bar\psi)\le\kappa d_{\mathcal Z}(\zeta,0)
	\end{align*}
	for all $\psi\in B_{\mathcal Y}(\bar \psi\text{; }\alpha_1)$ and $\zeta\in B_{\mathcal Z}(0\text{; }\alpha_2)$ 
satisfying $\zeta\in\Phi(\psi)$.
\end{Definition}

Notice that applying the definition with $\zeta = 0$ we obtain that $\bar \psi$ is the unique solution 
of the inclusion $0 \in \Phi(\psi)\cap B_{\mathcal Y}(\bar \psi\text{; }\alpha_1)$. In particular, $\bar u$ is a strict  
local minimizer for problem \eqref{ocp1}-\eqref{const}.

In the next assumption we introduce a restriction on the set of admissible perturbations, call it $\Gamma$, which is
valid for the  remaining part of this section.

\begin{Assumption}  \label{Ape}
For a fixed positive constant $C_{pe}$, the admissible perturbation $\zeta = (\xi,\eta,\rho) \in \Gamma \subset \mathcal Z$ 
satisfy the restriction 
\begin{equation} \label{pertbound}
         \|\xi \|_{L^r(Q)}\leq C_{pe} .
\end{equation}
\end{Assumption}

For any $u \in \mathcal U$ and $\zeta \in \Gamma$ we denote by $(y_u^\zeta, p_u^\zeta,u)$ a solution 
of the first two equations in \reff{s1per}.
Using \reff{semilin} in Theorem~\ref{estsemeq} we obtain the existence of a constant $K_y$ such that
\begin{equation} \label{E3.15}
        \|y_u^\zeta\| _{L^\infty(Q)}  \leq K_y \quad \forall u \in \mathcal U 
         \;\; \forall \zeta \in \Gamma.
\end{equation}
Then for every $u \in {\mathcal U}$, every admissible disturbance $\zeta$, and the corresponding solution $y$ of
the first equation in \reff{s1per} it holds that $(y_u^\zeta(x,t),u(t)) \in R := [- K_y , K_y] \times [u_a, u_b]^m$.

\begin{Remark} \label{Rlip}
We apply the local properties in  Assumption \ref{exist.2} to the interval $[-K_y, K_y]$, and denote by $\bar C$ 
a constant that majorates the bounds and the Lipschitz constants of $f$, $L_0$ and $L_1$
and their first and second derivatives with respect to $y \in [-K_y, K_y]$.
\end{Remark}

By increasing the constant $K_y$, if necessary, we may also estimate the adjoint state:
\begin{equation*}
        \|p_u^\zeta\| _{L^\infty(Q)}  \leq K_y (1 + \| \eta \|_{L^r(Q)}) \quad \forall u \in \mathcal U 
         \;\; \forall \zeta \in \Gamma.
\end{equation*}
This follows from Theorem \ref{mainex} with $\alpha = - \frac{\partial f}{\partial y}(x,t,y_u^\zeta)$ and with
$\frac{\partial L}{\partial y}(x,t,y_u^\zeta,u)$ at the place of $u$.

The main result of this the paper follows.

\begin{Theorem}\label{Ssr}
Let assumption 
\ref{A4}(B$_0$) be fulfilled for the reference solution 
$\bar \psi = (\bar y,\bar p, \bar u)$ of $0 \in \Phi(\psi)$.
Then the mapping $\Phi$ is strongly metrically subregular at 
$(\bar \psi,0)$. More precisely, there exist $\alpha_n, \kappa_n > 0$ such that 
for all $\psi \in \mathcal Y$ with $\| u-\bar u\|_{L^1(0,T)^m} \leq \alpha_n$ and $\zeta\in \Gamma$ satisfying $\zeta\in\Phi(\psi)$,
the following inequality is satisfied:
\begin{align}
& \!\!\!\!\!\!\!\!\!\!\!\! \label{EEy}  \| \bar u-u \|_{L^1(Q)}+\|y_{\bar u}-y^{\zeta}_u\|_{L^2(Q)}+\| p_{\bar u}-p^{\zeta}_u\|_{L^2(Q)}
\le \kappa_n \Big(\max_{1\leq j\leq m}\|\rho_j \|_{L^{\infty}(0,T)}+\|\xi\|_{L^2(Q)}+\|\eta\| _{L^2(Q)}\Big).
\end{align}
\end{Theorem}

To prove Theorem \ref{Ssr}, we need some technical lemmas.

\begin{Lemma}(\cite[Lemma 18]{CJV2022b}) \label{T4.1}
Let $u\in \mathcal U$ be given and $v\in L^r(0,T)^m$, $\xi,\eta\in L^r(Q)$. 
Consider solutions $y_u$, $p_u$, $z_{\bar u, v}$ and $y_u^\xi$,$p_u^\eta$, $z_{\bar u,v}^\xi$ of the equations
\begin{align}\label{s1}
\left\{ \begin{array}{cll}
\mathcal Ly+f(\cdot,\cdot,y)&=&\langle g,u\rangle,\\
\mathcal L^*p-\frac {\partial H}{\partial y}(\cdot,\cdot,y_u,p,u)&=&0,\\
\mathcal L_0z+f_y(\cdot,\cdot,y_u)z&=&\langle g,v\rangle,
\end{array} \right. \
\left\{ \begin{array}{cll}
\mathcal Ly+f(\cdot,\cdot,y)&=&\langle g,u\rangle+\xi,\\
\mathcal L^*p-\frac {\partial H}{\partial y}(\cdot,\cdot,y_u^\xi,p,u)&=&\eta,\\
\mathcal L_0z+f_y(\cdot,\cdot,y^\xi_u)z&=&\langle g,v\rangle,
\end{array} \right.
\end{align}
Here, $\mathcal L_0$ is defined as $\mathcal L$, but on the domain \reff{EDL} with $y_0 = 0$. Then for every $ s\in [1,\frac{n+2}{n})$
there exist constants $K_s,K_2,R_2>0$, independent of $\zeta \in \Gamma$, 
such that the following inequalities hold
\begin{align}
&\|y^\xi_u - y_u\| _{L^2(Q)} \le  C_2 \|\xi\|_{L^2(Q)},\label{E4.3.2}\\
&\| z^\xi_{u,v} - z_{u,v}\|_{L^2(Q)} \le K_2
\|\xi\|_{L^{r}(Q)}\|z_{u,v}\| _{L^2(Q)},\label{E4.4}\\
 &\| z^\xi_{u,v} - z_{u,v}\|_{L^s(Q)} \le K_s
 \|\xi\|_{L^{2}(Q)}\|z_{u,v}\| _{L^2(Q)},\label{E4.4b}\\
&\| p^\eta_u - p_u\|_{2}\le R_2(\|\xi\|_{L^2(Q)}+\| \eta\| _{L^2(Q)}),\label{E4.5.a2}
\end{align}
where $ C_2 $ is the constant given in \eqref{wl2}.
\end{Lemma}


 \begin{Lemma}
Let $u\in \mathcal U$ and $y_u$, $p_u$ be the corresponding state and adjoint state. 
Further, let $y_u^\zeta$ and $p^\zeta_u$ be solutions to the perturbed state and adjoint equation in \reff{s1per}
for the control $u$.
\begin{enumerate}

\item Let the function $L_1$ in the objective functional be independent of $y$. There exists a constant $C>0$,
independent of $\zeta \in \Gamma$, such that for all $v\in L^r(0,T)^m$, the following estimate holds: \begin{align}
\Big\vert  \int_Q&\Big \langle \frac {\partial H}{\partial u}(x,t,y_u,p_u)-
\frac {\partial H}{\partial u}(x,t,y^{\zeta}_u,p^{\zeta}_u),v\Big\rangle \dx\dt\Big\vert  \le  C(\|\xi\| _{L^2(Q)}+
\|\eta\| _{L^2(Q)})\| z_{u,v}\| _{L^2(Q)}.
\label{esti2}
\end{align}

\item There exists a constant $\tilde C>0$,
independent of $\zeta \in \Gamma$, such that for all $v\in L^r(0,T)^m$, the following estimate holds:\begin{align}
\Big\vert  \int_Q&\Big\langle\frac {\partial H}{\partial u}(x,t,y_u,p_u)-
\frac {\partial H}{\partial u}(x,t,y^{\zeta}_u,p^{\zeta}_u),v\Big\rangle\dx\dt\Big\vert\le \tilde C(\| \xi\|_{L^2(Q)}
       +\|\eta\| _{L^2(Q)})\| v\|_{L^1(0,T)^m}.
\label{esti2.1}
\end{align}
\end{enumerate}
\label{Impest}
\end{Lemma}

\begin{proof}{}
We begin with integrating by parts
\begin{align*}
&\Big\vert  \int_Q \Big \langle \frac {\partial H}{\partial u}(x,t,y_u,p_u)-\frac {\partial H}{\partial u}(x,t,y^{\zeta}_u,p^{\zeta}_u)
,v\Big \rangle \dx\dt\Big\vert \le\Big\vert  \int_Q\Big[\frac{\partial L_0}{\partial y}(x,t,y_u)z_{u,v}-
\frac{\partial L_0}{\partial y}(x,t,y^{\zeta}_u)
z_{u,v}^{\zeta}\Big]\dx\dt\Big\vert\\
&+\Big\vert\int_Q \Big\langle \frac{\partial L_1}{\partial y}(x,t,y_u),v\rangle z_{u,v}-
\langle \frac{\partial L_1}{\partial y}(x,t,y_u^\zeta),v\Big\rangle z_{u,v}^\zeta \dx\dt\Big\vert +
\Big\vert  \int_Q\Big\langle L_1(x,t,y_u)-L_1(x,t,y^\zeta),v\Big\rangle \dx\dt   \Big\vert\\
&+\Big \vert \int_Q  \eta z^\zeta_{u,v}\dx\dt\Big\vert=I_1+I_2+I_3+I_4.
\end{align*}
For the first term we use the H\"older inequality and the mean value theorem to estimate
\begin{align*}
I_1&\le\int_Q\Big\vert  \frac{\partial L_0}{\partial y}(x,t,y_u)-\frac{\partial L_0}{\partial y}(x,t,y^{\zeta}_u)\Big\vert  
\,\big\vert z_{u,v}\big\vert  \dx\dt+\int_Q\Big\vert  \frac{\partial L_0}{\partial y}(x,t,y^{\zeta}_u) \Big\vert   
\,\big\vert  z_{u,v}-z_{u,v}^{\zeta}\big\vert  \dx\dt\\
& \le \Big\| \frac{\partial^2 L_0}{\partial y^2}(x,t,y_{\theta})\Big\|_{L^\infty(Q)}\|   y^\zeta_{u} - y_{u}\|_{L^2(Q)}\| z_{u,v}\| _{L^2(Q)}\\
& +K_s \Big\|  \frac{\partial L_0}{\partial y}(x,t,y^{\zeta}_u)\Big\|_{L^{s'}(Q)}\| \xi \|_{L^2(Q)} \| z_{u,v}\|_{L^2(Q)},
\end{align*}
where $L^{s'}$ is the dual space to $L^s$.
By the mean value theorem, Assumption \ref{exist.2}, \eqref{unibound}, \eqref{E4.3.2} and \eqref{aeqestLs},  we can infer the existence of a constant $B_1>0$ such that 
\begin{equation}
I_1\le B_1 \|   
\xi\|   _{L^2(Q)}\|   z_{u,v}\|   _{L^2(Q)} .
\end{equation}
The second term is estimated by using Assumption \ref{exist.2}, \eqref{unibound}, H\"older's inequality, and \eqref{E4.4}:
\begin{align*}
I_2 &\leq \Big\vert\int_Q \Big\langle \frac{\partial L_1}{\partial y}(x,t,y_u)- 
\frac{\partial L_1}{\partial y}(x,t,y_u^\zeta),v\Big\rangle  z_{u,v} \dx\dt\Big\vert+
\Big\vert \int_Q \Big\langle  \frac{\partial L_1}{\partial y}(x,t,y_u^\zeta), v \Big\rangle  \Big[z_{u,v}-z_{u,v}^\zeta \Big] \dx\dt 
\Big \vert\\
&\le K_2 \max_{1\leq i\leq m} \Big\| \frac{\partial^2 L_{1,i}}{\partial y^2}(x,t,y_{\theta_i})v_i \Big\|_{L^\infty(Q)}
\| y_u-y^\zeta_u\|_{L^2(Q)} \| z_{u,v}\|_{L^2(Q)}\\
&+K_s \Big\|  \Big\langle \frac{\partial L_1}{\partial y}(x,t,y^{\zeta}_u),v\Big\rangle \Big\|_{L^{s'}(Q)}\| \xi \|_{L^2(Q)} 
\| z_{u,v}\|_{L^2(Q)}
\end{align*}
By the mean value theorem, Assumption \ref{exist.2}, \eqref{unibound}, \eqref{E4.3.2} and \eqref{aeqestLs}, we can infer the existence of a constant $B_2>0$ such that
\begin{equation}
I_2\leq B_2 \|   \xi\|   _{L^2(Q)}\|   z_{u,v}\|   _{L^2(Q)}.
\end{equation}
Applying the mean value theorem $m$ times, we obtain for the third term 
\begin{align*}
 I_3 & \leq \Big\vert  \int_Q\langle L_1(x,t,y_u)-L_1(x,t,y^\zeta),v\rangle \dx\dt \Big\vert\\
 & \leq \max_{1\leq j\leq m}\Big \|  \frac{\partial L_{1,j}}{\partial y}(x,t,y_{\theta_j}) 
\Big \|_{L^\infty(Q)} \| y_u-y^\zeta_u  \|_{L^\infty(0,T,L^2(\Omega)}) \|v\|_{L^1(0,T)^m}
\end{align*}
and infer by Assumption \ref{exist.2}, \eqref{unibound},  \eqref{aeqestLs} and \eqref{E4.3.2}, the existence of a constant $B_3>0$ with
\begin{equation*}
I_3\leq B_3 \|   
\xi\|_{L^2(Q)}\|v\|_{L^1(0,T)^m}.
\end{equation*}
For the last term, we estimate by Assumption \ref{exist.2}, \eqref{aeqestLs}, \eqref{unibound}, \eqref{E4.4}
and \eqref{E4.3.2}
\begin{align*}
I_4\leq \|\eta \|_{L^2(Q)}(\| z_{u,v} \|_{L^2(Q)}+\| z^\zeta_{u,v}-z_{u,v} \|_{L^2(Q)})\leq 
(1+K_2C_{pe}) \|\eta \|_{L^2(Q)}\| z_{u,v} \|_{L^2(Q)}
\end{align*}
and define $B_4:=1+K_2C_{pe}$.
If the function $L_1$ in the objective functional is independent of $y$, the term $I_3$ does not appear and the first 
estimate \eqref{esti2} holds for $C:=4 \max_{1\leq i\leq 4} B_i$. For the other case, \eqref{esti2.1}, we use that 
by Theorem \ref{mainex} and Lemma \ref{l2l1time} it holds
\begin{align*}
\|z_{u,v}\|_{L^2(Q)}&\leq  2\exp \Big(\Big\| \frac{\partial f}{\partial y}(\cdot, \cdot,\bar y(\cdot))\Big\|_{L^\infty(Q)}\Big) 
           \max _{1\leq j \leq m}\|g_j\|_{L^2(\Omega)}\| v\|_{L^1(0,T)^m}
\end{align*} 
and define $\tilde C$ in a similar way.
\end{proof}


\begin{proof}{ of Theorem \ref{Ssr}}
We select $\alpha <\tilde\alpha_0$ according to Lemma \ref{good}. 
Let $\zeta=(\xi,\eta,\rho)\in\mathcal Z$ and $\psi=(y^{\zeta}_u,p^{\zeta}_u,u)$ with $\| u-\bar u\|_{L^1(0,T)^m}\leq\alpha$ be
such that  $\zeta\in\Phi(\psi)$, i.e. 
\begin{align*}
\left\{ \begin{array}{cll}
\xi&=&\mathcal Ly^{\zeta}_u+f(\cdot,\cdot,y^{\zeta}_u) - u,\\
\eta&=&\mathcal L^*p_u^{\zeta}-\frac {\partial H}{\partial y}(\cdot,\cdot,y^{\zeta}_u,p_u^{\zeta},u),\\
\rho&\in& \int_\Omega\frac {\partial H}{\partial u}(x,\cdot,y^{\zeta}_u,p_u^{\zeta})\dx + N_{\mathcal U}(u).
\end{array} \right.
\end{align*}
Let $y_u$ and $p_u$ denote the solutions to the unperturbed problem with respect to $u$, i.e.
\[
\langle u,g\rangle=\mathcal Ly_u+f(\cdot,\cdot,y_u)\textrm{ and } 0=\mathcal L^*p_u-
\frac {\partial H}{\partial y}(\cdot,\cdot,y_u,p_u,u).
\]
By Lemma \ref{T4.1}, there exists $ C_2 ,R_2>0$ independent of $\psi$ and $\zeta$ such that
\begin{align}\label{mt1}
\|y^{\zeta}_u-y_{u}\|_{L^2(Q)}&+\|p^{\zeta}_u-p_{u}\| _{L^2(Q)}\le ( C_2 +R_2)\Big(\|\xi\|_{L^2(Q)}+\|\eta\|_{L^2(Q)}\Big).
	\end{align}
By the definition of the normal cone, $\rho\in \int_\Omega \frac {\partial H}{\partial u}(x,\cdot,y^{\zeta}_u,p^{\zeta}_u)\dx+
 N_{\mathcal U}(u)$
 is equivalent to
\begin{align*}
0\ge \int_Q\Big\langle \rho-\frac {\partial H}{\partial u}(x,t,y^{\zeta}_u,p^{\zeta}_u),w-u\Big\rangle\dx\dt\ \ \forall w\in \mathcal U.
\end{align*}
We conclude for $w=\bar u$, 
\begin{align*}
0&\ge \int_Q\Big \langle \frac {\partial H}{\partial u}(x,t,y_u,p_u),u-\bar u \Big\rangle\ \dx\dt+\int_Q\Big\langle \rho+
\frac {\partial H}{\partial u}(x,t,y_u,p_u)-\frac {\partial H}{\partial u}(x,t,y^{\zeta}_u,p^{\zeta}_u)
,\bar u-u \Big\rangle \dx \dt \nonumber\\
&\ge J'(u)(u-\bar u)-\max_{1\leq j\leq m}\|\rho_j\| _{L^{\infty}(0,T)}\|\bar u-u\| _{L^1(0,T)^m}\\
& \quad -\Big\vert  \int_Q \Big\langle \frac {\partial H}{\partial u}
(x,t,y_u,p_u)-\frac {\partial H}{\partial u}(x,t,y^{\zeta}_u,p^{\zeta}_u),\bar u-u\Big\rangle \dx\dt\Big\vert.
\end{align*}
By Lemma \ref{Impest}, we have an estimate on the third term. Since $\| u-\bar u\|_{L^1(0,T)^m}<\alpha$, 
we estimate by Lemma~\ref{good} and Lemma~\ref{Impest}
 \begin{align*}
\|u-\bar u\|   _{L^1(0,T)^m}^2\bar \gamma_0&\le J'(u)(u-\bar u)\\
&\le \tilde C \Big(\|   \xi\|_{L^2(Q)}+\|\eta\|_{L^2(Q)}\Big)
 \| u-\bar u\| _{L^1(0,T)^m}+\max_{1\leq j\leq m}\|\rho_j\| _{L^{\infty}(0,T)}\|\bar u-u\| _{L^1(0,T)^m}
\end{align*}
and consequently for an adapted constant, denoted in the same way
\begin{equation*}
\|\bar u-u\|_{L^1(0,T)^m}\le \tilde C\Big(\max_{1\leq j\leq m}\|\rho_j\| _{L^{\infty}(0,T)}+
\|\xi\| _{L^2(Q)}+\|\eta\| _{L^2(Q)}\Big).
\end{equation*}
To estimate the states, by Lemma \ref{l2l1time}, we use the estimate for the controls and obtain
\begin{align}
\|y_{\bar u}-y_u\| _{L^{2}(Q)}&\le 2\exp (\| \frac{\partial f}{\partial y}(\cdot,\cdot, \bar y(\cdot))\|_{L^\infty(Q)}) 
           \max _{1\leq j \leq m}\|g_j\|_{L^2(\Omega)}\|\bar u-u \|_{L^1(0,T)^m}\label{stc}.
\end{align}
Thus, for a constant again denoted by $\tilde C$
\begin{equation*}\label{Eststate}
\|   y_{\bar u}-y_u\| _{L^{2}(Q)}\le \tilde  C \Big(\max_{1\leq j\leq m}\|\rho_j\| _{L^{\infty}(0,T)}+\|\xi\|_{L^2(Q)}+\|\eta\| _{L^2(Q)}\Big).
\end{equation*}
Next, we realize that by Lemma \ref{T4.1} and \eqref{Eststate}
\begin{align*}
\|y_{\bar u}-y_u^\zeta\|   _{L^2(Q)}&\le \| y_{\bar u}-y_u \|_{L^2(Q)}+\|y_{u}-y_u^\zeta \|_{L^2(Q)}\\
&\le \max\{\tilde C, C_2 \}
\Big(\max_{1\leq j\leq m}\|\rho_j\| _{L^{\infty}(0,T)}+\|\xi\|_{L^2(Q)}+\|\eta\| _{L^2(Q)}\Big).
\end{align*}
Using $\|p_{\bar u}-p_u\| _{L^2(Q)}\le  C_2 \| y_{\bar u}-y_u\|_{L^2(Q)}$ and \eqref{E4.5.a2}, the same estimate holds
 for the adjoint state
\begin{align*}
\|p_{\bar u}-p_u^\zeta \|_{L^2(Q)}&\le \| p_{\bar u}-p_u \| _{L^2(Q)}+\|p_{u}-p_u^\zeta \|_{L^2(Q)}\\
&\le ( C_2 \max\{\tilde C, C_2 \}+R_2)
 \Big(\max_{1\leq j\leq m}\|\rho_j\| _{L^{\infty}(0,T)}+\|\xi\|_{L^2(Q)}+\|\eta\| _{L^2(Q)}\Big),
\end{align*}
subsequently we define $\kappa:= C_2 \max\{\tilde C, C_2 \}+R_2$. 
\end{proof}

To obtain results under Assumption \ref{A4} for $k\in\{1,2\}$, we need additional restrictions. 
We either don't allow perturbations $\rho$ (appearing in the inclusion in \reff{s1per}) or they need to satisfy 
\begin{equation}
\rho = \mu \sigma
\label{perturbeq}
\end{equation}
where $\mu=\int_\Omega g\dx \in \mathbb R^m$ and $\sigma \in W^{1,2}(0,T)$ with $\sigma(T)=0$. 

\begin{Theorem} \label{SsHr}
Let some of the assumptions 
$(A_1),(B_1)$ and $(A_2),(B_2)$ be fulfilled for the reference solution $\bar \psi = (\bar y,\bar p, \bar u)$ of $0 \in \Phi(\psi)$. 
Further, for $(A_1), (A_2)$ let the function $L_1$ in the objective functional be independent of $y$. For $(B_1), (B_2)$ 
let $L_1$ be affine with respect to $y$. In addition, the set $\Gamma$ of feasible perturbations is restricted to such 
$\zeta \in \Gamma$ for which 
the component $\rho$ is either zero or satisfies \eqref{perturbeq}. 
The numbers $\alpha_n$, $\kappa_n$ and $\varepsilon$ are as in Theorem \ref{Ssr}.
Then the following statements hold for $n \in \{ 1,2,3\}$:

1. Under Assumption \ref{A4}, cases $(A_1)$ and $(B_1)$, the estimation
\begin{align*}
 &\| \bar u-u \|_{L^1(0,T)^m}+\|y_{\bar u}-y^{\zeta}_u\|_{L^2(Q)}+\| p_{\bar u}-p^{\zeta}_u\|_{L^2(Q)}
\le \kappa_n \Big(\|\frac{ d\sigma }{dt} \|_{L^{2}(0,T)}+\|\xi\|_{L^2(Q)}+\|\eta\| _{L^2(Q)}\Big),
\end{align*}
hold for all $u\in \mathcal U$ with $\|y_u-\bar y\|_{L^\infty(Q)}<\alpha_n$, in the case of $(A_1)$, or 
$\|u-\bar u \|_{L^1(0,T)^m}<\alpha_n$ in the case ($B_1$), and for all $\zeta\in \Gamma$ satisfying $\zeta\in\Phi(\psi)$.

2. Under Assumption \ref{A4}, cases $(A_2)$ and $(B_2)$, the estimation
\begin{equation*}
\|\bar y-y^{\zeta}_u\|_{L^2(Q)}+\|\bar p-p^{\zeta}_u\|_{L^2(Q)} \le 
        \kappa_n \Big(\|\frac{ d\sigma }{dt} \|_{L^{2}(0,T)} +\|\xi\| _{L^2(Q)} + \|\eta\| _{L^2(Q)}\Big)
\end{equation*}
hold for all $u\in \mathcal U$ with $\|y_u-\bar y\|_{L^\infty(Q)}<\alpha_n$, in the case of $(A_2)$, or 
$\|u-\bar u \|_{L^1(0,T)^m}<\alpha_n$ in the case ($B_2$), and for all $\zeta\in \Gamma$ satisfying $\zeta\in\Phi(\psi)$.
\end{Theorem}

\begin{proof}{}
If the perturbation $\rho\in L^2(0,T,H^{-1}(\Omega))$ satisfies \eqref{perturbeq}, it holds
\begin{align*}
\int_0^T \langle \rho,u-\bar u\rangle\dt&=\nu\int_0^T \langle \sigma,u-\bar u\rangle\dt=\int_0^T \int _{\Omega} \sigma\langle g,u-\bar u\rangle \dx\dt\\
&=\int_0^T\int_\Omega (\mathcal L z_{\bar u,u-\bar u}+f_y(\cdot,t,y_{\bar u})z_{\bar u,u-\bar u})\sigma\dx \dt\\
&=\int_0^T \int_\Omega (-\frac{d \sigma}{dt}+f_y(\cdot,t,y_{\bar u})\sigma)z_{\bar u,u-\bar u}\dx\dt.
\end{align*}
Thus we can estimate
\begin{align*}
\Big\vert  \int_0^T \langle \rho,u-\bar u\rangle \dt\Big\vert&\le K(\|\frac{d \sigma}{dt}\| _{L^2(0,T)}+
\|f_y(x,t,y_{\bar u})\| _{L^\infty(Q)}\| \sigma\|_{L^2(0,T)})\| z_{\bar u,u-\bar u}\|_{L^2(Q)}\\
&\le K(\| \frac{d \sigma}{dt}\| _{L^2(0,T)}+
C_2\|f_y(x,t,y_{\bar u})\| _{L^\infty(Q)}\| \frac{d \sigma}{dt}\|_{L^2(0,T)})\| z_{\bar u,u-\bar u}\|_{L^2(Q)}.
\end{align*}
Under Assumptions $(A_1)$, $(B_1)$, we can proceed as in the proof of Theorem \ref{Ssr} using Lemma \ref{good} and 
\eqref{esti2} in Lemma \ref{Impest}, to infer the existence of constants $\alpha_1,\kappa_1>0$ such that
\begin{equation*}
\|\bar u-u\|_{L^1(0,T)^m}\le \kappa_1\Big(\|\frac{ d\sigma }{dt} \| _{L^{2}(Q)^m}+\|   \xi\|   _{L^2(Q)}+\| \eta\| _{L^2(Q)}\Big),
\end{equation*}
for all $u\in \mathcal U$ with $\|y_u-\bar y\|_{L^\infty(Q)}<\alpha_1$ or $\|u-\bar u \|_{L^1(0,T)^m}<\alpha_1$ 
depending on the assumption.
By standard estimates, using \eqref{E2.14.2}, there exists a constant $ E >0$, such that
\begin{align*}
&\| y_{\bar u}-y_u\| _{L^2(Q)}+\| p_{\bar u}-p_u\|_{L^2(Q)}\le  E \|    y_{\bar u}-y_u\| _{L^2(Q)}\le 2 
E\|z_{u,u-\bar u}\|_{L^2(Q)}\\
&\le 2 \kappa_1 E\Big(\| \frac{ d\sigma }{dt}  \| _{L^{2}(0,T)}+\|\xi\| _{L^2(Q)}+\|\eta\| _{L^2(Q)}\Big),
\end{align*}
for all $u\in \mathcal U$ with $\|y_u-\bar y\|_{L^\infty(Q)}<\alpha_1$ or $\|u-\bar u \|_{L^1(0,T)^m}<\alpha_1$ 
depending on the assumption.
From here on, we can proceed as in the proof of Theorem \ref{Ssr} and redefine the constant $\kappa_1>0$ accordingly.
Finally, by similar reasoning, under Assumption $(A_2)$,  $(B_2)$ with Lemma \ref{good} and Lemma \ref{Impest}, one 
obtains the existence of a constant $\kappa_2>0$ such that
\begin{align*}
&\| y_{\bar u}-y_u\|_{L^2(Q)}+\| p_{\bar u}-p_u\|_{L^2(Q)}\le \kappa_2\Big(\|\frac{ d\sigma }{dt} \|_{L^{2}(0,T)}+
\|\xi\| _{L^2(Q)}+\|\eta\|_{L^2(Q)}\Big),
\end{align*}
for all $u\in \mathcal U$ with $\|y_u-\bar y\|_{L^\infty(Q)}<\alpha_2$ or $\|u-\bar u \|_{L^1(0,T)^m}<\alpha_2$.
Again, proceeding as in Theorem \ref{Ssr} and increasing the constant $\kappa_2$ if needed, proves the claim.
\end{proof}

\begin{Remark} \label{Rnonlin}
Theorems \ref{Ssr} and \ref{SsHr} concern  perturbations which are functions of $x$ and $t$ only. On the other hand, 
\cite[Theorem ]{CDK} suggests that SMSr implies a similar stability property under classes of perturbations that depend 
(in a non-linear way) on the state and control. We refer to \cite[Section 5]{CJV2022b} for a detailed discussion on this. 
By straight forward adaptations, the results therin hold also for the problem considered in this paper.
\end{Remark}

\section*{Appendix} \label{secA1}
%

A proof of the following lemma can be found in \cite[Lemma 3.5]{CDJ2022} or \cite[Lemma 3.5]{CMR}.

\begin{Lemma} \label{L3.1}
Let $X=L^\infty(Q)$ or $L^2(Q)$. Given $\bar u \in \Uad$ with associated state $\bar y$, there exists a constant $B_X>0$ 
such that the following estimate holds
\begin{equation}
\|y_{\bar u + \theta(u - \bar u)} - \bar y\|_{X} \le B_X\|y_u - \bar y\|_{X} \quad \forall \theta \in [0,1]\ \text{ and }\ \forall u \in \Uad.
\label{E3.14}
\end{equation}
\end{Lemma}
We prove the analogous statement for the adjoint-state. For an elliptic state equation, a proof is given in 
\cite[Lemma 3.7]{CDJ2022}.

\begin{Lemma} \label{L3.2}
Let $X=L^\infty(Q)$ or $L^2(Q)$. Given $\bar u \in \Uad$ with associated state $\bar y$ and adjoint-state $\bar p$, 
then there exists a constant $\tilde B_X>0$ such that
\begin{equation}
\|p_{\bar u + \theta(u - \bar u)} - \bar p\|_{X} \le \tilde B_X(\|y_u - \bar y\|_{X} +\| u-\bar u\|_{L^1(0,T)^m}^{\frac{1}{r}}),
\label{E3.16}
\end{equation}
for all $\theta \in [0,1] \text{ and }u \in \Uad$. If the function $L_1$ in the objective functional is independent of $y$, 
then there exists a constant $\tilde B_X>0$ 
such that
\begin{equation}
\|p_{\bar u + \theta(u - \bar u)} - \bar p\|_{X} \le \tilde B_X\|y_u - \bar y\|_{X},
\label{E3.16.2}
\end{equation}
for all $\theta \in [0,1] \text{ and }u \in \Uad$.
\end{Lemma}

\begin{proof}{}
Let us prove \eqref{E3.16}. Given $u \in \mathcal U$ and $\theta \in [0,1]$, let us denote 
$u_\theta = \bar u + \theta(u - \bar u)$, $y_\theta = y_{u_\theta}$, and $p_\theta = p_{u_\theta}$. 
Subtracting the equations satisfied by $p_\theta$ and $\bar p$ we get with the mean value theorem
\begin{align*}
&-\frac{d}{dt}(p_\theta - \bar p)+\mathcal A^ *(p_\theta - \bar p) + \frac{\partial f}{\partial y}(x,t,\bar y)(p_\theta - \bar p) 
= \frac{\partial L}{\partial y}(x,t,y_\theta,u_\theta) - \frac{\partial L}{\partial y}(x,t,\bar y,\bar u)\\ 
&+ \Big[\frac{\partial f}{\partial y}(x,t,\bar y) - \frac{\partial f}{\partial y}(x,t,y_\theta)\Big]p_\theta\\
& = \frac{\partial L_0}{\partial y}(x,t,y_\theta) - \frac{\partial L_0}{\partial y}(x,t,\bar y)+
\Big\langle  \frac{\partial L_1}{\partial y}(x,t,y_\theta) - \frac{\partial L_1}{\partial y}(x,t,\bar y), u_\theta\Big\rangle\\
&+\Big\langle \frac{\partial L_1}{\partial y}(x,t,\bar y) ,u_\theta-\bar u\Big\rangle+ 
\Big[\frac{\partial f}{\partial y}(x,t,\bar y) - \frac{\partial f}{\partial y}(x,t,y_\theta)\Big]p_\theta\\
&=\frac{\partial^2 L_0}{\partial y^2}(x,t,y_{\vartheta_1})(y_\theta-\bar u)+
\sum_{1\leq j\leq m}\frac{\partial^2 L_{1,j}}{\partial^2 y}(x,t,y_{\vartheta_j}) u_{j,\theta}(y_{\theta}-\bar y)\\
&+\Big\langle \frac{\partial L_1}{\partial y}(x,t,\bar y) ,u_\theta-\bar u\Big\rangle+
\frac{\partial^2 f}{\partial y^2}(x,t, y_{\vartheta_{m+1}})(\bar y-y_\theta)p_\theta,
\end{align*}
where $y_{\vartheta_{i}}=\bar y+\vartheta_{i}(y_\theta - \bar y)$ for some measurable functions 
$\vartheta_{i}:Q \longrightarrow [0,1]$, $i=0,...,m+1$. Now, we can apply Theorem \ref{mainex} and Remark \ref{Rlip} 
to conclude from the above equation the existence of a constant $C_X>0$ such that
\begin{align*}
\|p_\theta - \bar p\| _{X}&\le C_X(\|y_\theta - \bar y\|_{X}+ \|u-\bar u\|_{L^r(Q)})\\
&\leq C_X(B_X\|y_u-\bar y\|_{X}+\vert \Omega \vert (2M_{\mathcal U})^{\frac{r-1}{r}}\|u-\bar u\|_{L^1(0,T)^m}^{\frac{1}{r}}).
\end{align*}
Defining $\tilde B_X:=C_X(B_X+\vert \Omega \vert (2M_{\mathcal U})^{\frac{r-1}{r}})$, with $B_X$ being 
the constant from Lemma \ref{L3.1}, concludes the proof of the first claim. The second claim follows by the same argument and the fact that the right hand side of the equation satisfied by $p_\theta - \bar p$ does not depend on $L_1$.
\end{proof}

Below we shall use the next lemma, the proof of which can be found for linear elliptic equations in 
\cite[Lemma 2.3]{CDJ2022} and for parabolic equations in \cite{CJV2022b}.

\begin{Lemma} \label{estLs}
Let $u\in L^{r}(Q)$ and $0\leq \alpha\in L^\infty(Q)$. Let $y_u$ be the unique solution of \eqref{lin.1} and 
let $p_u$ be a solution of the problem
\begin{align}
		\left\{\begin{array}{l}
		-\frac{\partial p}{\partial t}+\mathcal A^*p+\alpha p = u\  \text{ in }\ Q,\\
		p=0 \text{ on } \Sigma,\ p(\cdot,T)=0\ \text{ on } \Omega.
	\end{array} \right.
	\label{adjlin}
\end{align}
Then, for any $s_n\in [1,\frac{n+2}{n})$ there exists a constant $C_{s'_{n}}>0$ independent of $u$ and $\alpha$ such that 
	\begin{equation}
		\max\{\|y_u\|_{L^{s_n}(Q)},\| p_u\|   _{L^{s_n}(Q)}\} \leq C_{s'_{n}}\|u\|_{L^{1}(Q)}.
	\label{aeqestLs}	
	\end{equation}
Here $s'_n$ denotes the H\"older conjugate of $s_n$.
\end{Lemma}

\begin{proof}{ of Lemma \ref{biglemmat}}
The second variation of the objective functional is given by Theorem \ref{T3.1}. Let us denote $u_\theta$, $y_\theta$, 
and $\varphi_\theta$ as in the proof of Lemma \ref{L3.2}. From \eqref{E3.5} we obtain that
\begin{align*}
&\vert  [J''(\bar u + \theta(u - \bar u)) - J''(\bar u)](u - \bar u)^2\vert\\
&\le \int_Q\Big\vert  \Big[\frac{\partial^2L}{\partial y^2}(x,t,y_\theta,u_\theta) - 
\frac{\partial^2L}{\partial y^2}(x,t,\bar y,\bar u)\Big] z_{u_\theta,u - \bar u}^2\Big\vert\dx\dt+ 
\int_Q\Big\vert  (\bar\varphi - \varphi_\theta)\frac{\partial^2f}{\partial y^2}(x,t,y_\theta)z_{u_\theta,u - \bar u}^2\Big\vert  \dx\dt\\
&+ \int_Q\Big\vert  \bar\varphi\Big[\frac{\partial^2f}{\partial y^2}(x,t,\bar y) - 
\frac{\partial^2f}{\partial y^2}(x,t,y_\theta)\Big]z_{u_\theta,u - \bar u}^2\Big\vert  \dx\dt\\
& + \int_Q\Big\vert  \Big[\frac{\partial^2L}{\partial y^2}(x,t,\bar y)- 
\bar\varphi\frac{\partial^2f}{\partial y^2}(x,t,\bar y)\Big](z^2_{u_\theta,u - \bar u} - z^2_{\bar u,u - \bar u})\Big\vert  \dx\dt\\
&+2\int_Q\Big\vert   \Big\langle \frac{\partial L_1}{\partial y}(x,t,y_{\theta})-
\frac{\partial L_1}{\partial y}(x,t,\bar y), z_{u_\theta,u - \bar u}(u-\bar u)\Big\rangle\Big\vert  \dx\dt\\
&+2\int_Q \Big\vert  \Big\langle\frac{\partial L_1}{\partial y}(x,t,\bar y), (z_{u_\theta,u - \bar u}-
z_{\bar u,u - \bar u})(u-\bar u)\Big\rangle\Big\vert  \dx\dt\\
& = I_1 + I_2 + I_3 + I_4+I_5+I_6.
\end{align*}
The first term, $I_1$, can be estimated as
\begin{align*}
I_1&\leq\int_Q\Big\vert  \Big[\frac{\partial^2L_0}{\partial y^2}(x,t,y_\theta) - 
\frac{\partial^2L_0}{\partial y^2}(x,t,\bar y)\Big] z_{u_\theta,u - \bar u}^2\Big\vert\dx\dt\\
&+ \int_Q\Big\vert \Big[ \Big\langle \frac{\partial^2 L_1}{\partial y^2}(x,t,y_{\theta})-
\frac{\partial^2 L_1}{\partial y^2}(x,t,\bar y),u_\theta\Big\rangle  +
\Big \langle \frac{\partial^2 L_1}{\partial y^2}(x,t,\bar y),u_\theta-\bar u \Big\rangle\Big] z_{u_\theta,u - \bar u}^2\Big\vert\dx\dt\\
&=I_{1,1}+I_{1,2}+I_{1,3}.
\end{align*}
For the first two terms, we deduce from Assumption \ref{exist.2}, Remark \ref{rebound}, 
Remark \ref{Rlip}, \eqref{E3.14}, \eqref{E2.14.2} and \eqref{E2.15.2}, 
that for every $\rho_{1,i} > 0$ there exists $\varepsilon_{1,i} > 0$ such that
\[
I_{1,i} \le \rho_{1,i}\|z_{\bar u,u - \bar u}\|^2_{L^2(Q)}\quad  \text{if}\quad \|y_u - \bar y\|_{L^\infty(Q)} < \varepsilon_{1,i}, \ \ i=1,2.
\]
For $I_{1,3}$ we estimate under Assumption \ref{exist.2}, Remark \ref{rebound}, Remark \ref{Rlip}, \eqref{E3.14}, 
\eqref{l2l1timeest}, \eqref{E2.14.2}, \eqref{E2.15.2}, that for $\|y_u - \bar y\|_{C(\bar Q)} $ sufficiently small
\begin{align*}
&\int_Q \Big\vert \Big \langle \frac{\partial^2 L_1}{\partial y^2}(x,t,\bar y),u_\theta-
\bar u \Big\rangle z_{u_\theta,u - \bar u}^2\Big\vert\dx\dt\\
&\leq \| z_{u_\theta,u-\bar u}\|_{L^\infty(0,T,L^2(\Omega)}^2\|u_\theta-\bar u\|_{L^1(0,T)^m} \max_{j=1,...,m}\Big\| 
\frac{\partial^2L_{1,j}}{\partial y^2}(x,t,\bar y)\Big\|_{L^\infty(Q)}\\
&\leq \frac{9}{4} \vert \Omega\vert^{\frac{1}{2}} \exp(\|\frac{\partial f}{\partial y}(\cdot,\cdot,y)\|_{L^\infty(Q)}) 
\| \bar y -y_u\|_{L^\infty(Q)}\|u-\bar u\|_{L^1(0,T)^m}^2 \max_{j=1,...,m}\| g_j\|_{L^\infty(Q)} 
\max_{j=1,...,m}\Big\| \frac{\partial^2L_{1,j}}{\partial y^2}(x,t,\bar y)\Big\|_{L^\infty(Q)}.
\end{align*}
We can therefore infer, that for every $\rho_{1,3} > 0$ there exists $\varepsilon_{1,3}> 0$ such that
\[
I_{1,3} \le \rho_1\|u-\bar u\|_{L^1(0,T)^m}^2\quad \text{if}\quad \|y_u - \bar y\|_{L^\infty(Q)} < \varepsilon_{1,3}.
\]
For the term $I_2$, we first consider the general case. 
Using Assumption \ref{exist.2}, Remark \ref{rebound}, Remark \ref{Rlip}, \eqref{E3.14}, \eqref{E3.16}, \eqref{E2.14.2} 
and \eqref{E2.15.2}, we find for any $\rho_2>0$ a $\varepsilon_2>0$ such that
\begin{align*}
I_2 &\le\frac{9}{4}\bar{C}\tilde BC_r(2M_{\mathcal U})^\frac{r-1}{r})\| u-\bar u\|_{L^1(0,T)^m}^{\frac{1}{r}} 
\| z_{\bar u,u - \bar u}\|_{L^2(Q)}^2\leq 
\rho_2\|z_{\bar u,u - \bar u}\|^2_{L^2(Q)}\quad \text{if}\quad \| u - \bar u\|   _{L^1(0,T)^m} < \varepsilon_2.
\end{align*}
In the case that $\frac{\partial L_1}{\partial y}\equiv 0$, we deduce from Assumption \ref{exist.2}, Remark \ref{rebound}, 
Remark \ref{Rlip}, \eqref{E3.14}, \eqref{E3.16}, \eqref{E2.14.2} and \eqref{E2.15.2}, 
that for every $\rho_2 > 0$ there exists $\varepsilon_2> 0$ such that
\[
I_2 \le \rho_j\|z_{\bar u,u - \bar u}\|^2_{L^2(Q)}\quad  \text{if}\quad \|y_u - \bar y\|_{L^\infty(Q)} < \varepsilon_2.
\]
For the term $I_3$ we deduce from Assumption \ref{exist.2}, Remark \ref{rebound}, Remark \ref{Rlip}, \eqref{E3.14}, 
\eqref{E3.16}, \eqref{E2.14.2} and \eqref{E2.15.2}, 
that for every $\rho_3 > 0$ there exists $\varepsilon_3 > 0$ such that
\[
I_3 \le \rho_3\|z_{\bar u,u - \bar u}\|^2_{L^2(Q)}\quad  \text{if}\quad \|y_u - \bar y\|_{L^\infty(Q)} < \varepsilon_3
\]
For $I_4$ we define $\psi:=
z_{\bar u,u - \bar u}-z_{u_\theta,u - \bar u}$. $\psi$ solves the equation
\begin{align*}
\frac{d \psi}{dt}+\mathcal{A}\psi + \frac{\partial f}{\partial y}(x,t,y_{\bar u})\psi &= 
\Big[ \frac{\partial f}{\partial y}(x,t,y_{u_\theta})-\frac{\partial f}{\partial y}(x,t,y_{\bar u})\Big]z_{u_\theta,u - \bar u}
= \frac{\partial^2f}{\partial y^2}(x,t,y_\vartheta)(y_{\bar u} - y_{u_\theta})z_{u_\theta,u - \bar u},
\end{align*}
where we used the mean value theorem to infer the existence of a function $\vartheta$ such that the above holds.
We apply \eqref{wl2} to $\psi$ and estimate
\begin{align*}
I_4 &\le \Big\| \frac{\partial^2L}{\partial y^2}(x,t,\bar y)- 
\bar\varphi\frac{\partial^2f}{\partial y^2}(x,t,\bar y)\Big\|_{L^\infty(Q)}\|   z_{u_\theta,u - \bar u} + 
z_{\bar u,u - \bar u}\|_{L^2(Q)}\| z_{u_\theta,u - \bar u} - 
z_{\bar u,u - \bar u}\|   _{L^2(Q)}\\
& \le \frac{3M_{2}}{2}\| y_\theta - \bar y\| _{L^\infty(Q)}\|   z_{\bar u,u - \bar u}\|   _{L^2(Q)}^2.
\end{align*}
Then by Assumption \ref{exist.2}, Remark \ref{rebound}, Remark \ref{Rlip}, \eqref{E3.14}, \eqref{E2.14.2} 
and \eqref{E2.15.2}, for every $\rho_4 > 0$ there exists $\varepsilon_4>0$ such that 
\begin{equation}
I_4\le \rho_4\| z_{\bar u,u - \bar u}\|^2_{L^2(Q)} \quad \text{if}\quad \|y_u - \bar y\|   _{L^\infty(Q)} < \varepsilon_4.
\end{equation}
The term $I_5$, can be estimate similar as $I_{1,3}$, therefore under under Assumption \ref{exist.2}, Remark \ref{rebound}, 
Remark \ref{Rlip}, \eqref{E3.14}, \eqref{l2l1timeest}, \eqref{E2.14.2}, \eqref{E2.15.2}, for every $\rho_5> 0$ 
there exists $\varepsilon_5>0$ such that 
\[
I_5 \le \rho_5\|u-\bar u\|_{L^1(0,T)^m}^2\quad  \text{if}\quad \|y_u - \bar y\|_{L^\infty(Q)} < \varepsilon_5.
\]
To estimate $I_6$, we select $s$ as in Lemma \ref{estLs} and apply \eqref{aeqestLs} to $\psi$ and estimate for 
$\|y_u - \bar y\|_{L^\infty(Q)}$ sufficiently small
\begin{align*}
I_6 &\leq
2 M_{\mathcal U}^{\frac{s'-1}{s'}}\max_{j=1,..,m}\Big\| \frac{\partial L_{1,j}}{\partial y}(\cdot,\bar y(\cdot))\Big\|_{L^\infty(Q)}
\| u - \bar u\|_{L^{1}(0,T)^m}^{\frac{1}{s'}}\|  z_{u_\theta,u - \bar u}-z_{\bar u,u - \bar u}\|_{L^s(Q)}\\
&\le 2M_{\mathcal U}^{\frac{s'-1}{s'}}\|\frac{\partial^2f}{\partial y^2}(x,t,y_\vartheta)\|_{L^\infty(Q)} \max_{j=1,..,m}\Big\| 
\frac{\partial L_{1,j}}{\partial y}(\cdot,\bar y(\cdot))\Big\|_{L^\infty(Q)}\| u - \bar u\|_{L^{1}(0,T)^m}^{\frac{1}{s'}}
 \|y_{u_\theta} - \bar y\| _{L^2(Q)}\| z_{\bar u_{\theta},u - \bar u}\| _{L^2(Q)}.
\end{align*}
Thus, depending on the chosen estimation, under Assumption \ref{exist.2}, Remark \ref{rebound}, 
Remark \ref{Rlip}, \eqref{E3.14}, \eqref{E2.14.2} and \eqref{E2.15.2}, for every $\rho_6> 0$ there exists 
$\varepsilon_6>0$ such that
\[
I_6 \le \rho_6\| z_{\bar u,u - \bar u}\| _{L^2(Q)}^2\quad  \text{if}\quad \|u-\bar u\|_{L^1(0,T)^m} < \varepsilon_6.
\]
We remark that by \eqref{clr}, 
\[
\|u-\bar u\|_{L^1(0,T)^m}<\frac{\varepsilon^r}{(C_r\vert \Omega\vert^{\frac{1}{r}}\|g\|_{L^\infty(\Omega)^m}
\|u_a-u_b\|_{L^\infty(0,T)^m}^\frac{r-1}{r})^{r}}
\]
implies $\|y_u - \bar y\|_{L^\infty(Q)} < \varepsilon$. \\
If the function $L_1$ in the objective functional is independent of $y$, the problematic parts in the terms  $I_{1}, I_2, I_5 $ 
and $I_6$ are absent. Further, if the function $L_1$ is affine with respect to $y$, 
the problematic parts in the terms  $I_{1}, I_2, I_5 $ and $I_6$ are either absent or can be estimated under 
the condition that $\| u - \bar u\|_{L^{1}(0,T)^m}$ is sufficiently small. If this is not the case, we only obtain item \ref{drei} 
of the Lemma \ref{biglemmat}.
Depending on the terms in the objective functional, by taking $\rho_i$ so small that $I_i < \frac{\rho}{6}$ 
for every $i\in \{1,..,6\}$ and setting $\varepsilon = 
\min_{1 \le i \le 6}\varepsilon_i$, we complete the proof.
\end{proof}



\end{document}